\documentclass{gtpart} 

\usepackage [all]{xy}       
 
\title{Quantum characteristic classes and the Hofer metric}

\author{Yasha Savelyev}
\givenname{Yasha} 
\surname{Savelyev}
\address{Stony Brook University\\\newline
Department of Mathematics\\
Stony Brook\\NY 11790\\USA}
\email{yasha@math.sunysb.edu}
\urladdr{http://www.math.sunysb.edu/~yasha}
  
\volumenumber{12}
\issuenumber{4}
\publicationyear{2008}
\papernumber{52}
\startpage{2277}
\endpage{2326}

\doi{}
\MR{}
\Zbl{}   


\received{9 February 2008}
\revised{18 July 2008}
\accepted{5 June 2008}
\published{2 September 2008}
\publishedonline{2 September 2008}
\proposed{Leonid Polterovich}
\seconded{Yasha Eliashberg, Simon Donaldson}
\corresponding{}
\editor{Liz}
\version{}

\arxivreference{arXiv:0709.4510}
\arxivpassword{cswnr}

\let\xysavmatrix\xymatrix
\def\xymatrix{\disablesubscriptcorrection\xysavmatrix}
\AtBeginDocument{\let\bar\wbar\let\tilde\wtilde\let\hat\what}
\makeop{pr}
\makeop{id}
\makeop{pt}
\makeop{ev}
\makeop{reg}
\makeop{tr}
\def\const{\mathrm{const}}
\makeop{Ham}
\makeop{Hor}
\def\sect{\mathrm{sect}}
\def\vert{\mathrm{vert}}
\makeop{PD}
\makeop{PGW}
\makeop{Hol}
\makeop{hol}
\def\BHam{B\mathrm{Ham}}

\newcommand{\ls}{\Omega \mathrm {Ham} (M, \omega)}
\newcommand{\freels}{L \mathrm {Ham}}

\DeclareMathOperator{\area}{area}

\DeclareMathOperator{\coker}{coker}
\DeclareMathOperator{\Aut}{Aut}
\DeclareMathOperator{\Vol}{Vol}

\DeclareMathOperator{\Bord}{Bord}

\newtheorem {theorem} {Theorem} [section]  

\newtheorem{lemma}[theorem] {Lemma} 

\newtheorem {question}  [theorem] {Question}
\newtheorem {proposition}  [theorem]{Proposition} 
\newtheorem {corollary}[theorem]  {Corollary} 
\newtheorem {axiom}{Axiom}  

\theoremstyle{definition}
\newtheorem {definition} [theorem] {Definition} 
\newtheorem {remark} [theorem] {Remark}
\newtheorem {example} [theorem]  {Example}
\newtheorem {notation}[theorem] {Notation} 

\begin{document}

\begin{asciiabstract}
Given a closed monotone symplectic manifold M, we define certain
characteristic cohomology classes of the free loop space LHam(M,omega)
with values in QH_*(M), and their S^1 equivariant version. These
classes generalize the Seidel representation and satisfy versions of
the axioms for Chern classes. In particular there is a Whitney sum
formula, which gives rise to a graded ring homomorphism from the ring
H_*(Omega Ham(M, omega),Q), with its Pontryagin product to
QH_{2n+*}(M) with its quantum product.  As an application we prove an
extension to higher dimensional geometry of the loop space
LHam(M,omega) of a theorem of McDuff and Slimowitz on minimality in
the Hofer metric of a semifree Hamiltonian circle action.
\end{asciiabstract}

\begin{abstract} 
Given a closed monotone symplectic manifold $M$, we define certain
characteristic cohomology classes of the free loop space
$L\mathrm{Ham}(M, \omega)$ with values in $QH_* (M)$, and their $S^1$
equivariant version. These classes generalize the Seidel
representation and satisfy versions of the axioms for Chern
classes. In particular there is a Whitney sum formula, which gives
rise to a graded ring homomorphism from the ring $H_{*} ( \Omega
\mathrm{Ham}(M, \omega), \mathbb{Q})$, with its Pontryagin product to
$QH_{2n+*} (M)$ with its quantum product.  As an application we prove
an extension to higher dimensional geometry of the loop space $L
\mathrm{Ham}(M, \omega)$ of a theorem of McDuff and Slimowitz on
minimality in the Hofer metric of a semifree Hamiltonian circle
action.
\end{abstract}

\maketitle

\section{Introduction}         
The topology and geometry of the group $ \Ham(M, \omega)$ of Hamiltonian
sym\-plect\-o\-morph\-isms of a symplectic manifold $M$ has been intensely studied by
numerous authors. This is an infinite-dimensional manifold with a remarkable
bi-invariant Finsler metric induced by the Hofer norm. It is of
foundational importance in symplectic geometry and Hamiltonian
mechanics.  As of now the deepest insights into the topology and Hofer geometry
of $ \Ham(M, \omega)$ come from Gromov--Witten invariants and related quantum and Floer
homology constructions. Still, rather little general information is known.

We define some general invariants, which will be used to study $\Ham(M, \omega)$. The Hofer geometry will serve a unifying role, and ultimately
is what allows us to compute the invariants in some cases. These computations can
in turn be used to get  Hofer geometric and topological information.  What
follows is a rapid review  of Hofer geometry. 

\subsection {Hofer geometry and Seidel representation}  Given a smooth function $H _{t}\co M
\to \mathbb{R}$, $0 \leq t\leq 1$,  there is  an associated time dependent
Hamiltonian vector field $X _{t}$, $0 \leq t \leq 1$,  defined by 
\begin{equation*}  
\omega (X _{t}, \cdot)=dH _{t} ( \cdot).
\end{equation*}
The vector field $X _{t}$ generates a path $\gamma_{t}$, $0 \leq t \leq1$, in
$\text
{Diff}(M, \omega)$. Given such a path $\gamma_t$, its endpoint $\gamma_1$ is
called a
Hamiltonian symplectomorphism. The space of Hamiltonian symplectomorphisms
forms a group, denoted by $\Ham(M, \omega)$.

In particular the path $\gamma_t$ above lies in $ \Ham(M, \omega)$. It
is well known 
that any path $\{\gamma _{t} \}$ in $\Ham(M, \omega)$ with $\gamma_0=\id$
arises in this way (is generated by $H_t\co  M \to \mathbb{R}$). Given such a path
$\{\gamma_t\}$, the \emph{Hofer length} $L (\gamma _{t})$ is defined by
\begin{equation*}L (\gamma_{t}):= \int_0 ^{1} \max (H ^{\gamma} _{t}) -\min (H
^{\gamma}_t) \,dt,
\end{equation*}
where $H ^{\gamma} _{t}$ is a generating function for the path
$ \gamma_{0} ^{-1} \gamma _{t},$ $0\leq t \leq 1$.
The Hofer distance $\rho (\phi, \psi)$ is  defined by taking the
infimum of the Hofer length of paths from $\phi $ to $\psi$.
It is a deep theorem that the resulting metric is nondegenerate (see
Hofer~\cite{H} and Lolonde and McDuff~\cite{LM}). This gives $ \Ham(M, \omega)$ the structure of a Finsler
manifold. We will be more concerned with a related measure of the path
\begin{equation*} L ^{+} (\gamma_{t}):= \int_0 ^{1} \max (H ^{\gamma} _{t}),
\end{equation*} where $H ^{\gamma} _{t}$ is in addition normalized by the
condition 
\begin{equation*} \int _{M} H ^{\gamma _{t}}=0.
\end{equation*}
Let $\gamma\co  S ^{1} \to \Ham(M, \omega)$ be a subgroup with generating
Hamiltonian $H$.  Let $F _{\max}, F _{\min}$ denote the
maximum and minimum level sets of $H$ (these are fixed by $\gamma$). We say
that $\gamma$ is \emph{semifree} at $F _{\max}$, respectively $F _{\min}$, if all
the nonzero weights of the linearized action of $\gamma$ on the normal bundles
to $F _{\max}$, respectively $F _{\min}$, are $-1$, respectively $+1$. 
To define these weights we take an $S^1$
equivariant orientation preserving identification of the normal bundle to $F _{\max}$ at $x \in F
_{\max}$ with $ \mathbb{C} ^{m}$, for some $m$, which splits into $\gamma$
invariant $1$--complex-dimensional subspaces $N _{k_i}$, on which $\gamma$ is
acting by \begin{equation} \label {eq.weights} v \mapsto e ^{-2\pi ik_i \theta}v.
\end{equation}
These $k_i$ are then defined to the weights of the circle action $\gamma$.
Here is one theorem that does give some general
information about topology and  geometry of $ \Ham(M, \omega)$. 
\begin{theorem}  [McDuff--Slimowitz--Tolman \cite{MSlim,MT}] \label
{thm.mcduff.slim}  Let $\gamma$ be a Hamiltonian circle action on a symplectic
manifold $M$, which  is semifree at $F _{\max}$ and $F _{\min}$. Then $\gamma$
is length-minimizing in its homotopy class for the Hofer metric on $ \Ham(M, \omega)$.
\end{theorem}
 One of the motivating applications of this thesis is an
extension of this theorem to the higher-dimensional geometry of $ \Ham(M, \omega)$.
The starting point for this is the Seidel representation
defined in \cite{Seidel}. This is a homomorphism
 \begin {equation} \label {eqS} {S}\co \pi_1 ( \Ham(M,
  \omega)) \to QH _{2n} ^{ \times} (M),
\end {equation} where $QH _{2n} ^{ \times} (M)$
 denotes the group of multiplicative units of degree 2n in quantum homology $QH_*
 (M)$, and $2n$ is the dimension of $M$. This representation of
 $\pi_1 ( \Ham(M, \omega))$ is a powerful tool in
 understanding the symplectic geometry of the manifold $ (M, \omega)$,
 of Hamiltonian fibrations $X _{\phi}$ over $S^2$  associated to loops $
 \{\phi_t\}$ in $ \Ham(M, \omega)$,
 as well as the Hofer geometry and topology of the group $ \Ham(M,
 \omega)$. In particular, \fullref{thm.mcduff.slim} can be proved using the
 Seidel representation as is essentially done by McDuff and Tolman~\cite{MT}.

There is a completely natural extension of Seidel representation to certain
cohomology classes of the associated loop spaces of $ \Ham(M, \omega)$.
\subsection {Quantum characteristic classes} 
Consider the free loop space $\freels(M, \omega)$, which we will
abbreviate by $\freels$. We construct natural bundles 
\begin{equation*} 
\tilde{p}\co  U\to \freels 
\quad \text{and} \quad p\co  U^{\smash{S^1}} \to Q \equiv (\freels \times S ^{\infty})/S
^{1}. \end{equation*} The fiber over a loop $\gamma$ is modelled by a
Hamiltonian fibration $\pi\co 
X _{\gamma} \to S^2$, with fiber $M$, associated to  the
loop $\gamma$ as follows:
\begin{equation} \label {eq.X}  X _{\gamma}= (M \times D^2_0 ) \cup 
(M \times D^2_\infty )/\sim
\end{equation}
where the equivalence relation $\sim$ is: $ (x, 1, \theta)_0 \sim (
\gamma_\theta (x),1,\theta)_\infty$. Here $D ^{2}_0$ and $D ^{2} _{\infty}$
are two names for the unit disk $D ^{2} \subset \mathbb{C}$ and $ (r, 2\pi
\theta)$ are polar coordinates on $D$.

Let  $p\co  P \to B$ be a bundle obtained by pullback of either $ \tilde{p}\co
 U \to \freels$ or $p\co  U^{\smash{S^1}} \to Q$, where $B$ is a closed oriented smooth manifold.
%
The bundle $P$ comes with a natural deformation class of families of
symplectic forms $ \{\Omega_b\}$ on the fibers $ \{X_b\}$.
We will define
characteristic classes $$c ^{q}_k (P) \in H ^{k} (B, QH_* (M)),$$ by counting
the number of fiber-wise or vertical $J$--holomorphic curves passing through
certain natural homology classes in $P$. Here $k$ is the degree of the class and the
superscript $q$ stands for quantum to distinguish it from the Chern classes $c_k$.
\begin{remark} \label {remark.quant.classes} Recall that the top Chern class
of a complex vector bundle is its Euler class, whose Poincare dual is represented by the self
intersection of the zero section. The classes $c ^{q} _{k} (P)$ are
also
in a sense described by the self intersection of a natural homology class in
$P$, playing the role of the zero section, except that the classical
intersection is always empty and instead one keeps track of
``instanton (or quantum) corrections" to this
self-intersection, coming from the presence of vertical $J$--holomorphic spheres.
This is the motivation for the name quantum characteristic class (cf
Vafa~\cite{Vafa}).
\end{remark}
\begin{remark} Michael Hutchings discovered maps
\begin{equation}\label {H} \pi_k ( \Ham(M, \omega)) \to \text {End} _{k - 1}
(QH_* (M)),
\end{equation} where $ \text {End} _{k - 1} (QH_* (M))$ denotes the additive
group of endomorphisms of the quantum homology of degree $k-1$.  These maps
generalize the Seidel representation. In fact his project is much more extensive \cite{Michael,MH}.
Hutchings defines \eqref{H} as a kind of a ``family'' or ``higher'' continuation map in
Floer homology. Dusa McDuff \cite{private} suggested the following approach to
these maps. Glue together $ M \times
(\mathbb{CP}^{k}\setminus B)$ with $M \times \bar {B}$, where $B$ is an open
$2k$ dimensional ball, by using a map $f\co  S ^{2k-1} \to \Ham(M, \omega)$
as a clutching map. We get a certain Hamiltonian fibration with fiber $M$
over $
\mathbb{CP} ^{k}$ and may presumably define a parametrized variant of the Seidel
representation
which takes into account the natural family of holomorphic curves in $
\mathbb{CP} ^{k}$. We are motivated by this idea but our approach is slightly
more
abstract, which also leads to a close relationship with the Hofer metric.
We expect that our invariants will coincide with Hutchings' for
cycles in the loop space $\ls$ induced by homotopy classes in $\pi_k (
\Ham(M, \omega))$ (cf \fullref{section.homotopy.groups} and also
\fullref {4}).
\end{remark} 

\begin{definition} \label {definition.sum} Given  fibrations $P _{f_1}, P
_{f_2}$ over $B$, induced by maps $$f_1, f_2\co  B \to \freels(M, \omega)$$
we define their  sum $P _{f_1} \oplus P _{f_2}$ to be  $P _{f_2 \cdot f_1}$, where $f_2
\cdot f_1\co  B \to \freels(M, \omega)$
is the pointwise product of the maps $f_1, f_2$ induced by the topological group
structure of $\freels(M, \omega)$.
\end{definition}
We'll show in \fullref{structure group} that these
fibrations have a natural structure group $\mathcal {F}$ and that $\freels(M, \omega)$ is
the classifying space of this structure group. Let $ \mathcal
{P} _{B, M}$ denote the group of isomorphism classes of fibrations $p: P
\to B$ with structure group $\mathcal {F}$ (ie the group of homotopy classes
of maps $f\co  B \to \freels$).

 We may now state the axioms satisfied by
our characteristic classes. For simplicity we assume here that the base $B$ is
connected.
\begin{definition} \label {def.quant.class} \emph{Quantum characteristic
classes} are a sequence
of functions $$c ^{q} _{k}\co  \mathcal {P} _{B, M} \to H ^{k} (B, QH_* (M)),$$ satisfying the following axioms:
 \begin{axiom} [Partial normalization]\label{Axiom 1} $c ^{q}_0 (P)= S ( [\gamma])$ if
the fiber of $p\co  P \to B$ is modelled on $X _{\gamma}$,
where $S$ is the Seidel
representation. Further if $P$ is trivial then $c ^{q}_k (P)=0$ for $k>0$.
\end{axiom}
\begin{axiom} [Functoriality] \label{functoriality}  If $g\co  B_1 \to B_2$ is
smooth, then $$ g ^{*} (c ^{q}_k (P_2))= c ^{q}_k ( g ^{*} (P_2)).$$
\end{axiom}
\begin{axiom} [Whitney sum formula] \label{thm.whitney.sum} If $P= P_1 \oplus
P_2$, then
\begin {equation*} c ^{q} (P) = c ^{q}(P_1) \cup c ^{q} (P_2),
\end {equation*} where $ \cup$ is the cup product of cohomology classes with
coefficients in the quantum homology ring $QH_* (M)$ and $c ^{q} (P)$ is the
total characteristic class
\begin{equation} \label {eq.total.class}
c^{q} (P)= c ^{q}_0 (P)+\ldots+ c ^{q}_m (P),
\end{equation}
where $m$ is the dimension of $B$.
(In practice, we mainly deal with the identity component of
$\freels$. In this case $c ^{q}_0 (P)=S ( [\gamma])$ is the
identity $ [M]$ in the quantum
homology ring and so we get an expression in \eqref{eq.total.class} analogous to the total Chern class.)
\end{axiom}
\end{definition}
 \begin{theorem} \label {thm.three.axioms} If $ (M, \omega)$ is a
closed monotone symplectic manifold, then there exist
natural nontrivial quantum characteristic classes $$c_k^{q}\co  \mathcal {P} _{B, M} \to H ^{k} (B, QH_* (M)).$$
\end {theorem}
We define these classes in \fullref{quantum.char.classes} 
and  prove in \fullref{axiom1and2} there that they satisfy Axioms \ref{Axiom 1}, \ref{functoriality} and
\ref{thm.whitney.sum}. 
\begin{remark} Throughout we work with the class of monotone symplectic manifolds  $(M,
\omega)$, ie those satisfying $$[\omega]= \const \cdot c_1 (M),$$ for a $\const
>0$. This condition insures that the relevant evaluation maps are pseudocycles.
It is likely that this condition can be removed in the definition of
QC classes by using methods of the virtual moduli
cycle. However a few properties of QC classes may potentially require
monotonicity, notably \fullref{4}. 
\end{remark}

\begin{remark}  We make no claim for uniqueness of these classes,
as there are not enough axioms here. We are missing something like a
normalization axiom (see for example Milnor and Stasheff~\cite{MilnorStasheff}). It would be
interesting to know if one can find a suitable substitute.
\end{remark}
 \subsection {Generalized Seidel representation}
For the following discussion we consider the based loop space $\Omega \Ham(M, \omega)$. This is a topological group with product induced by the
product in  $ \Ham(M, \omega)$, there is an induced product on homology,
the Pontryagin product, giving $H_* ( \Omega \Ham(M, \omega))$ the structure of a ring.
Let $f\co  B \to \freels$  be a map from a
smooth oriented closed
$k$--manifold, and $P _{f}$ the induced fibration. Define
\begin{equation*} 
\Psi (B,f) \equiv c ^{q} _{k} (P _{f}) (B) \in QH
_{2n+k}
(M). \end{equation*}
We will show that this induces a map $$\Psi\co 
H_* (\ls, \mathbb{Q}) \to QH _{2n+*} (M).$$ The Whitney sum formula (\fullref{thm.whitney.sum}) will imply that $\Psi$ is a graded ring homomorphism.
\begin{theorem} \label{main theorem} Let $ (M, \omega)$ be a closed, monotone
symplectic
manifold of dimension 2n. There is a natural graded ring homomorphism
\begin{equation*} 
{\Psi}\co  H _{*} (\ls,
\mathbb{Q}) \to QH_{*+2n} (M),
\end {equation*} where the
product on the right is the quantum product. 
\end{theorem}
In \fullref{quantum.char.classes} and \fullref{axiom1and2} we will describe 
these constructions and results in detail and give some computations and
applications.
\subsection[Applications from \ref{QC.and.Hofer}]{Applications from \fullref{QC.and.Hofer}}
 Given a map $$f\co B
\to \freels \text \quad \text {or} \quad f\co  B \to Q=(\freels(M, \omega)
\times S^\infty) /S ^{1},$$ where $B$ is as before, we call \begin{equation*} 
L^+(f) \equiv \max _{b \in B} L ^{+} (\gamma_b),
\end{equation*}
the \emph{positive max-length measure} of $f$, where $\gamma_b$ is either the
loop $f (b) \in \freels$ or the $S^1$ equivariant loop $f (b) \in Q$.
More precisely, in the second case let $q\co  \freels \times S^\infty \to Q$
denote the $S^1$ quotient map. Then $f (b) = q (\gamma_b, s_b)$ for some $(\gamma_b, s_b)$  and any two choices are related by an action of $S^1$ and
hence the corresponding loops $\gamma_b$ have the same length.

We define the \emph{virtual index} of a one
parameter subgroup $\gamma\co  S^1 \to \Ham(M, \omega)$ by
\begin{equation*} 
I (\gamma)=
\sum _{\substack {1 \leq i \leq n \\ k_i  \leq -1}} 2(|k_i|  -1),
\end{equation*}
where $k_i$ are the \emph{weights} of $\gamma$ at the
max level set $F_{\max}$ of $H$ (see \eqref{eq.weights}).

We shall see in \fullref{section.costr.eq.cycles} that a map $\hat{f}\co
Y \to \Ham(M, \omega),$ where $q\co  Y \to B$ is a principal $S^1$ bundle induces a cycle $f\co B \to
Q$. 
\begin{theorem} \label{main.lie.group} Let $ (M, \omega)$ be a compact monotone
symplectic manifold and let $ \hat{f}\co  Y \to \Ham(M, \omega)$ be
equivariant with respect to a right action by $\gamma\co  S^1 \to \Ham(M, \omega)$ on $ \Ham(M,
\omega)$, such that $I (\gamma)= \dim B$ and $e
^{1/2 \dim B} \neq 0$ (or $\dim B=0$) where $e$ is the
Euler class of the $S^1$ bundle $Y \to B$. Then
the induced cycle $f\co  B \to Q$ is essential, ie doesn't vanish in the oriented
bordism group $\Bord _{I (\gamma)} (Q)$, and moreover it minimizes the positive max-length
measure in its bordism class.
\end{theorem}
\begin{remark}  \label{remark.index} In \cite{U}
Ustilovsky gives a formula for the Hessian, ie  the ``second
variation formula'' for the Hofer length functional and its positive and negative
variants. We might try to define the \emph{index} of a Hofer geodesic $\gamma$
to be the dimension of the maximum subspace of the tangent space to $\gamma$ (in
$\freels$) on which
the corresponding Hessian for the positive Hofer length functional is
negative definite. This index could  well be infinite as we are
working on the
loop space of the infinite dimensional space $ \Ham(M, \omega)$.
However, \fullref{main.lie.group} suggests that at least for the geodesic
coming from a
circle action that satisfies hypotheses of the theorem the index must be
finite. The
heuristic argument for this, as well as for necessity of the virtual index
 condition of the theorem, is the following.
Up to the action of $S^1$, all the loops in the image
 $f(B)$ are of the form $ \hat{f} (y) \circ
\gamma$ for $y \in Y$, by our assumption that $ \hat{f}\co  Y \to \Ham(M, \omega)$ is $S^1$ equivariant.
Since the Hofer metric is bi-invariant all these loops have the same index as
$\gamma$. Moreover, we should get a certain vector bundle over the image $f (B)$ whose
fiber over the equivariant loop $ f (b)$, $b \in B$, is ``the'' maximum negative
definite subspace of the tangent space to $f (b)$, with respect to the
corresponding Hessian. This is slightly wrong as there is no way to
canonically pick out this negative definite subspace. However, we can fix such a
subspace
of the tangent space at $\gamma$ and then use the fact that all the other loops
are translates of $\gamma$ of the form $ \hat{f} (y) \circ \gamma$ up to
the action of $S^1$ to construct this bundle locally and glue to get a global bundle. 
Let's call this \emph{ND bundle}. If the rank of
this ND bundle, given by the index, is bigger than $\dim B$ we can
push the zero section off
of itself and then ``exponentiate'' to produce a deformation of the cycle $f: B
\to Q$ which reduces the max length measure; an
apparent contradiction. On the other hand if the index is equal to
$\dim B$, then there is an obstruction to reducing the
max length measure  by such a local move coming from the Euler class of the
ND bundle. Lastly, if the index is
strictly less than $ \dim B$ there is still an obstruction coming from
the Euler class but it is no longer in the top cohomology of $B$.  
Therefore while the cycle $f\co  B \to Q$ may minimize the max-length measure
locally and maybe even in its homotopy class it, it may be unreasonable to hope
that it is minimizing in the entire bordism class. 
\end{remark}
\begin{remark}  This heuristic
argument suggests that a necessary condition for minimality of $f\co  B \to
Q$ above is that the index is equal to $\dim B$. This condition is local;
on the other hand the conclusion of \fullref{main.lie.group} is global.
Nevertheless, to prove
it we compute the ``leading order'' contribution to the top quantum characteristic
class of the associated bundle $p\co  P_f \to B$,
in terms of the Euler class of a vector bundle  analogous to the ND
bundle above.
A bit more precisely, this bundle will be an obstruction bundle for a certain
moduli space of holomorphic curves (cf \fullref{section.lie.group}).
\end{remark}
 These remarks motivate the questions.
 \begin{question} Is the  index of $\gamma$ as
 defined above finite? Do the index of $\gamma$ and the virtual index
 of $\gamma$ coincide?  We say a bit on this topic in a sequel
 \cite{Sequel}.
 \end{question} 
\begin{example}[for \fullref{main.lie.group}] \label
{section.example.S3} Consider the  Lie group
homomorphism $\hat{f}\co  S^3 \to \Ham( \mathbb{CP} ^{n}, \omega)$, given by $$s
\cdot( [z_0, z_1,
\ldots, z_n])= [s (z_0, z_1), \ldots, z_n] \quad \text {for all } s \in S
^{3}=SU (2), \ \ [z_0, \ldots, z_n] \in \mathbb{CP}^n.$$ We can form an $S^1$
bundle $h\co S^3 \to S^2$ by taking the quotient of $SU (2)$ by the right
action of the diagonal $S^1$ subgroup $ \theta \mapsto (e ^{i \theta}, e
^{-i\theta})$. If we take  $\gamma\co  S^1 \to \Ham( \mathbb{CP}^n,
\omega)$ to be the subgroup \begin{equation*} e ^{i\theta} \cdot [z_0,z_1, \ldots, z_n]= [e
^{i\theta}z_0, e ^{-i\theta}z_1,z_2,, \ldots, z_n],
\end{equation*} acting on $ \Ham( \mathbb{CP} ^{n}, \omega)$ on the right
then the map $ \hat{f}$ is $S^1$ equivariant for the two actions.
The weights of $\gamma$ at the maximum $\max=
[1,0,0,\ldots]$ of the generating function $H$ are $-2, -1, -1, \ldots$
and so $I(\gamma)=2$. Thus, by \fullref{main.lie.group}   the associated
cycle $$f_h\co  S^2 \to Q$$ is essential and minimizes the max-length measure in its bordism class. 

But there is another
cycle we can assign to $ \hat{f}$. This is the cycle $$f\co  S^2 \to \ls$$
obtained from $ \hat{f}\co  S^3 \to \Ham( \mathbb{CP}^n, \omega)$ by slicing $S^3$ into a bouquet of circles (cf \fullref{section.homotopy.groups}). One can show that this cycle is essential by
more elementary methods (cf Kedra--McDuff~\cite{kedra-2005-9} and 
Reznikov~\cite{Rez}), but these arguments do not show that it minimizes the max-length
measure.
\end{example}

The only non trivial characteristic class of $P _{f}$ is $c_2 ^{q} (P _{f})$.
Computing this directly is difficult, but we may use the following theorem
proved in \fullref{proof.ref4}.    
\begin{theorem}  \label {4} Let $ (M,
\omega)$ be a spherically monotone, compact symplectic manifold,
$\hat{f}\co  S  
^{2k+1} \to \Ham(M, \omega)$ a smooth map, and $$f_h\co \mathbb{CP} ^{k}
\to Q, \quad f\co  S ^{2k} \to \ls$$ obtained from $ \hat{f}$ as in \fullref{section.homotopy.groups}. Then the only possibly nontrivial 
 characteristic classes of $P _{f_h}$  and $P _{f}$ in degree other than 0  are
 the top characteristic classes $c _{2k} ^{q} (P
 _{f_h})$, $c _{2k} ^{q} (P _{f})$ and
 \begin{equation} \label {eq.equality.q}\Psi (f_h, \mathbb{CP} ^{k})= \Psi
 ({f}, S ^{2k}) \in QH _{2n+2k} (M).
\end{equation}
\end {theorem}
Note that $ c ^{q} (P _{f_h})$ and $ c ^{q} (P _{f})$ are computed via PGW
invariants of two topologically very different fibrations, as $f$ and $f_h$ are not even
homologous in $Q$. So there is no obvious apriori reason for
\eqref{eq.equality.q} to hold. Using this as well as \fullref{thm.leading.term} and \fullref{ex.S3} 
we deduce that for our  $f\co  S ^{2} \to \ls$,
\begin{equation*}  \Psi (S ^{2}, f) = [-\pt] \otimes q ^{-m _{\max}}t ^{H
_{\max}} + \text {lower $t$--order terms} \in QH _{2n+2} ( \mathbb{CP} ^{n}),
\end{equation*}
where $m_{\max}=\sum
_{i} k_i=-2-(n-1)$ is the sum of the weights at the $\max$ and $H _{\max}$ is
the maximum of $H$. Using the above and \fullref{lower bound} we can
deduce the following. 
\begin{corollary} \label {main.example} The above map $f\co   S^2 \to \ls$ is
 minimal in its rational homology class for the max-length measure.
\end{corollary}
\begin{remark}  The crucial part of the above calculation is that $$
\hat{f}\co  SU(2) \to \Ham( \mathbb{CP}^n, \omega)$$ is
$S^1$--equivariant in an appropriate way, and so we
may apply \fullref{main.lie.group}. One may try to extend the calculation by
taking $$ \hat{f}\co  SU (n) \to \Ham( \mathbb{CP} ^{n-1}, \omega),$$
and consider some associated cycle $f_h\co  SU (n)/S^1 \to Q$. However, the
nonvanishing condition on the Euler class in \fullref{main.lie.group}, $e
^{1/2( \dim SU (n)-1)} \neq 0$ will never be
satisfied because of the topology of the group $SU(n)$, according to McDuff
\cite{private}.  There may of course be other examples, possibly not even coming
from Lie group actions.
\end{remark}

\subsection {Some questions}
 \begin{question} Does $f\co  S ^{2} \to \ls$
 remain
 minimal under the iterated Pontryagin product, with respect to  the max-length
 measure, ie is $$f ^{k} \co  (S^2) ^{k} \to \ls$$ minimal in its homology class?
\end{question}
A computation using \fullref{main theorem} shows that
the lower bounds coming from characteristic classes  (\fullref{lower
bound}) would grow to infinity but it is not clear if they stay sharp.

The following theorem is a slight reformulation of McDuff--Slimowitz \cite {MSlim}.
\begin{theorem}\label {thmdusa} Let $\gamma\co  S^1 \to \Ham(M, \omega)$ be a
 Hamiltonian circle action generated by a Morse Hamiltonian $H$. Suppose
$\gamma$ is a local minimum of the Hofer length functional. Then it is a global
minimum in its homotopy class.
\end{theorem}
 \begin{proof}[Proof (sketch)] It is well known that the max, min level sets of a
 Hamiltonian circle
 action are connected. Thus, since $H$ is Morse there is a unique max and min.
Consider the following theorem.
\begin{theorem} [McDuff--Lalonde \cite{ML1}]\label{theorem.dusa.lalonde}
Let ${H_t}$, $t \in [0,1]$ be a Hamiltonian defined on any
symplectic manifold $M$, and $\gamma = {\phi_t}$   the corresponding isotopy.
Assume that each fixed extremum of ${H_t}$ is isolated
among the set of fixed extrema. If $\gamma$ is a stable geodesic (ie a local
minimum of the length functional) there exist at least one fixed minimum
 $p$ and one fixed maximum $P$ at which the differential of the isotopy has no
 non constant closed trajectory in time less than 1.
\end{theorem}
In our case this says that when $\gamma$ is a local minimum of the Hofer length
functional and is generated by a Morse Hamiltonian the linearized flow at max and min
 corresponding to $\gamma$ has no nonconstant periodic orbits with period less
 than 1. This condition is called \emph{semifree} at max and min. On the other hand
 this puts us in position to apply the following theorem.
\begin{theorem} [McDuff--Tolman \cite{MT}] Let $\gamma$ be a Hamiltonian circle
action with semifree
maximal fixed point set and generating function $H$. Then there are classes
$a_B \in H_* (M)$ such that:
\begin{equation*} S (\gamma)= [F _{\text{max}}] \otimes q ^{-m _{\text{max}}} t
^{H_ {\text{max}}} +
\sum _{B \in H_2 ^{S}| \omega (B)>0} a _{B} \otimes q ^{-m _{\text{max}}- c_1 (B)} t
^{H_{\text{max}} -\omega _{B}}
\end{equation*}
\end{theorem}
Here $S$ is the Seidel representation of \eqref{eqS},  $H _{\text{max}}$
denotes the maximum value of $H$ and $F _{\max}$ denotes the max level set. This
expression implies that the positive Hofer length of
 the loop $\gamma$ is bounded below by $H _{\max}$ (cf \fullref{lower
 bound}).  Reversing $\gamma$ and applying the same theorem, we similarly get that
the negative Hofer length of $\gamma$ is bounded below by $-H _{\min}$. Together
this implies the Hofer length of $\gamma$ is bounded below by $H _{ \text
{max}}- H _{ \text {min}}$.
 \end{proof}
\begin{question} Can the condition on $H$ being Morse in \fullref {thmdusa}
be dropped or relaxed?
\end{question}
We can think of \fullref{thmdusa} and \fullref{main.lie.group} as  local to
global rigidity type of phenomena in $ \Ham(M, \omega)$. One may wonder
to what extent this can
be extended.  One  question which motivated this paper is the following.
\begin{question} Let $G$ be a
closed $k$--dimensional Lie group and
$h\co  G \to \Ham(M, \omega)$  a Lie group homomorphism (perhaps with finite
kernel).
Suppose $h$ is a local minimum for a ``natural volume functional'' induced
by the Hofer metric on $ \Ham(M, \omega)$. Is $h$ necessarily a global
minimum in its homotopy class? Homology class?
\end{question}
There are a few natural notions of volume in a Finsler manifold; one that is
often used is the Hausdorff $k$--measure but it may not be the easiest to work
with. We refer the reader to \'Alvarez and Thompson~\cite{Alvarez} for a
discussion of these notions.


\subsubsection* {Acknowledgements} This is part of the author's doctoral research
at SUNY Stony Brook. I would like to thank my advisor Dusa McDuff for
her patience and support both moral and mathematical, and numerous suggestions
in writing this paper. I am grateful to Blaine Lawson, Dennis Sullivan  and  
Aleksey Zinger for many useful discussions and comments and to an
anonymous referee for careful reading and some keen observations and comments.
\section {Preliminaries and setup}
\label {def of homomorphisms} 
In this section we describe constructions of certain natural fibrations, and
of Parametric Gromov--Witten invariants defined on the total spaces of
these fibrations.

Let $Q$ be the Borel $S^1$ quotient of $\freels$,
$Q=(\freels \times S ^{\infty})/S ^{1}$, where the action of
$S^1$ on $S ^{\infty}$ is by multiplication by $e ^{i \tau}$, for $\tau \in S ^{1}$
and on $\freels$ by $(\tau \cdot \gamma) (\theta)=\gamma (\theta + \tau)$. 
Let $q$ denote the quotient map
\begin{equation} \label {eq.q} q\co  \freels \times S ^{\infty} \to Q.
\end{equation}

\subsection {Fibrations over \texorpdfstring{$\freels$}{LHam} and \texorpdfstring{$Q$}{Q}} \label {setup}
There is a natural fibration over $\freels$: $$ \tilde{p}\co U\to
\freels$$
where 
\begin{equation} \label {eq8} U = \freels \times M \times D ^{2}_0
\cup \freels \times M \times D^2_\infty / \sim
 \end{equation} and the equivalence relation $\sim$ is: $ (\gamma, x, 1,
 \theta)_0 \sim (\gamma, \gamma_\theta (x),1,\theta)_\infty$. Here, $(r, 2 \pi
 \theta)$ are polar coordinates on $D^2$,
 and $\gamma_\theta$ denotes the element of the loop
 $\gamma$ at time $\theta$. The orientation on $M \times D ^{2}_0$ is taken to
 be the natural positive orientation and on $M \times D ^{2}_\infty$ is taken to
 be negative. There is a natural $S^1$ action on ${U}$
\begin{equation*} \tau \cdot (\gamma,x,r,\theta)  
_{0,\infty}= (\tau \cdot \gamma, x, r, \theta-\tau) _{0,\infty}
\end{equation*} 
where $\tau \in S ^{^1}$ and $(\tau \cdot \gamma)_ \theta=\gamma
_{\theta+\tau}$, ie the
standard $S^1$ action on the loop space. It can be quickly checked that this is
well defined under the equivalence relation $\sim$. Thus, the diagonal action
$\rho$ of $S^1$ on $\freels \times S^\infty$ lifts to a diagonal action $
\tilde{\rho}$ on the product fibration
  \begin{equation*} 
  \tilde{p} \times \id\co   {U} \times
S^\infty \to \freels \times S^\infty.
\end{equation*}
This gives a quotient
bundle 
\begin{equation*} 
p\co  U^{\smash{S^1}}= (U\times  S
^{\infty})/S ^{1} \to Q. \end{equation*}
 The fiber $X _{
q(\gamma,s)}$ of $U^{\smash{S^1}}$
over $q(\gamma,s)$ (see \eqref{eq.q}) is the total space of the Hamiltonian bundle
$X _{\gamma}$ (cf \eqref{eq.X}). 
We recall for the reader:
\begin{definition}  A \emph{Hamiltonian bundle} is a bundle with
 symplectic fiber, whose transition maps are Hamiltonian. A \emph{
Hamiltonian bundle map} is a bundle map which preserves the Hamiltonian bundle structure.
\end{definition}
\begin{remark} \label {remark.structure.group}  We show in \fullref{structure group} that the structure
group of $ \tilde{p}: U  
\to \freels$ over the component containing the loop $\gamma$ may be reduced to
the group $ \mathcal {F} ^{\gamma}$ of Hamiltonian  
bundle maps of the fiber $X _{\gamma}$, which are identity over $D ^{2}_0$ and
a neighborhood of $0 \in \smash{D ^{2}_\infty}$.  A very similar description holds for
the structure group of $p\co  \smash{U^{\smash{S^1}}} \to Q$, in particular it consists of
certain  Hamiltonian bundle maps. The groups $ \mathcal {F} ^{\gamma}$ are 
isomorphic for all $\gamma$ and we just refer to the groups as $\mathcal {F}$. 
As already mentioned in the Introduction and proved in \fullref{structure group}, the space $\freels$ is the classifying space for $
\mathcal {F}$. (More precisely, the component of the loop $\gamma$ in $\freels$ is the classifying space for $\mathcal {F} ^{\gamma}$.) We call a fiber bundle $p\co  P \to B$, with fiber having the structure of a Hamiltonian fibration
 $\pi\co  X \to S^2$ and structure group $ \mathcal {F}$ an \emph{$\mathcal
 {F}$--fibration}. The structure group of the bundle pulled back from $p\co 
 U^{\smash{S^1}} \to Q$ is also determined in \fullref{structure group} and it also
 consists of special Hamiltonian bundle maps. We will call both types of bundles simply by $ \mathcal
 {F}$--fibration, where there can be no confusion.
\end{remark}

 \subsection {Families of symplectic forms on an \texorpdfstring{$\mathcal
 {F}$--}{F-}fibration} \label {section.families}
Let $p\co  P \to B$ be an $ \mathcal {F}$--fibration, in the sense of above remark.
Fix an area form $\alpha$ on the base $S^2$ of $\pi\co  X \to S^2$ once and for all. Since
the fibers $M$  
are canonically oriented as symplectic manifolds and since the transition maps
of  $\pi\co  X \to S^2$ are Hamiltonian and hence preserve that orientation,
this induces an orientation $\sigma$
on the fibers $X$ of $P$, which is again preserved by the structure group $ \mathcal {F}$
of the bundle $P$. Thus, since $B$ is oriented $P$ inherits a well defined 
orientation.  
\begin {definition} \label{def.compatible}Let $\pi\co  X \to S^2$ be  a Hamiltonian
fibration with fiber $ (M, \omega)$.
We say that a symplectic form $\Omega$ on $X$ is
\emph{$\omega$--compatible} if it extends ${\omega}$ on the fibers. 
\end {definition}
 Let $ \mathcal {A}$ consist of all
  $\omega$--compatible symplectic forms $\Omega$ on $X$ inducing the orientation
  $\sigma$ (note, the cohomology class of $\Omega$ is not fixed).  Since $
  \mathcal {F}$ acts on $ \mathcal {A}$, we have the associated bundle
$k\co  K_P \to B$ with fiber $ \mathcal {A}$.
\begin {definition}  Let $p\co  P \to B$ be an  $ \mathcal {F}$--fibration.
 A family of symplectic structures  $ \{\Omega_b\}$ on $P$ is called
 \emph{admissible} if it is a section of $K_ {P}$. 
\end {definition}
\begin {lemma} \label{structures} The space of  admissible families $
\{\Omega_b\}$ on $p\co  P \to B$ is connected and nonempty.
\end {lemma}
\begin {proof} We show that the fiber $ \mathcal {A}$ of the bundle $K$ is at least
weakly contractible, ie has vanishing  homotopy groups. It will
follow from obstruction theory that the space of sections is connected and nonempty.

Let $h\co  S^k \to \mathcal {A}$ be a continuous map.  We denote $h(s)$ by
$\Omega_s$. Let $\Omega_0 \in \mathcal {A}$. The path 
$$ \Omega _{t,s}=t\Omega_0+ (1-t) \Omega_s, \quad t \in I= [0,1]$$ may not lie
in $ \mathcal {A}$, as $\Omega _{t,s}$ may be degenerate for some $t$, so we will need to
compensate. For $t$, $s \in I \times S^k$, let $\Hor ^{t,s}$ denote the
horizontal subbundle of $TX$ with respect to  $\Omega _{t,s}$, ie $\Hor ^{t,s}$
is the symplectic orthogonal to the vertical tangent bundle of $\pi\co  X \to S
^{2}$. 

 Let $\Omega _{t,s}^h$ denote the horizontal
part of $\Omega _{t,s}$, ie $\Omega _{t,s}^h$ is zero on the vertical
subbundle of $TX$ and coincides with
$\Omega _{t,s}$ on $\Hor ^{t,s}$.
Then $$\Omega _{t,s}^h= f_ {t,s}\cdot \pi^* (\alpha) \textrm  {, where } f
_{t,s}\co  X \to \mathbb{R} \textrm { is smooth.}$$ Recall that $\alpha$ is the
fixed area form on $S^2$.  Set
$$C=|\inf_ {t,s \in I \times S ^{k}
} (\inf_X f _{t,s})|+1$$   
and define  
\begin{align*}
\phi (t)&=\begin {cases} 0 &  \textrm{if  $t \in [0,1/3];$}\\
                        {3(t-1/3)} & \textrm{if $t \in [1/3, 2/3];$}\\
 1  & \textrm {if $t \in [2/3,1];$}\\
\end {cases}
\\
\eta (t)&= \begin{cases} 
                                                     3t & \textrm {if $t \in 
                                                     [0,1/3];$}\\
                                                                 1 & \textrm
                                                                 {if $t \in [1/3,2/3];$}\\
                                                                  - 3(t-2/3)+1 
                                                                  & \textrm {if
                                                                  $t \in [2/3,1].$}\\
\end{cases}
\end{align*}
  Consider the following homotopy of the map $h$: $$F (t,s)=\phi (t) \Omega +
  (1-\phi (t))\Omega_s + \eta (t)C\pi^*
 (\alpha).$$ Then $F (1, x)$ is the constant map to $\Omega_0$ and $F (0,x)= h (x)
 $. Since that $\Omega_s$ and $\Omega$ induce
the same orientation on $X$, $f _{0,s}, f _{1,s} > 0$. Using this, it is clear
that the form $F (t,x)$ is
nondegenerate on $X$ for every $t, x$, and so $F (t,x)$ is a map into $ \mathcal {A}$.
 Thus, all the homotopy groups of  $ \mathcal {A}$ vanish.      
 \end{proof}
This discussion shows
that we may choose an admissible family $ \{\Omega_b\}$
on $P$ and moreover any two such families are deformation equivalent. 
We will
now construct a special family that will be crucial in
applications to the Hofer metric.
As the first step we
define a family of symplectic forms $\{ \wwtilde{\Omega} ^{\infty} _{
\gamma} \}$ on $\freels  \times M \times D^2_\infty $: 
\begin{equation*}  \wwtilde{\Omega}_{\gamma} ^{\infty} (x,r, \theta) = \omega
+ d  \big( \eta (r) H^\gamma_ \theta (\gamma_0 ^{-1}x)\big) \wedge
d\theta - \max_x {H_\theta^\gamma (x})d\eta \wedge d\theta - \epsilon \cdot 2r
dr \wedge d\theta, \end{equation*} for an $\epsilon>0$. (Recall that $M \times
D^2_\infty$ has the negative orientation.) Here, $\smash{H ^{\gamma}_\theta}$ is the
generating Hamiltonian for $\gamma ^{-1} (0) \circ \gamma$, normalized so that $$\int
_{M}
H ^{\gamma}_\theta \omega ^{n}=0 $$ for all $\theta$, and  $\eta\co 
[0,1] \to [0,1]$ is a smooth function satisfying \begin{align*}0 &\leq \eta' (r)\\
\eta (r) &= \begin{cases} 1 & \text{if } 1 -\delta \leq r \leq 1 ,\\
r ^{2}  & \text{if } r \leq 1-2\delta,
\end{cases}\tag*{\hbox{and}}\end{align*}  for a small $\delta >0$. The last 2 terms 
are needed to make the sum nondegenerate. The following geometric notion will be
important to us for tying Hofer geometry with geometry of holomorphic curves.
\begin{definition} The \emph{area} of a Hamiltonian fibration $\pi\co X
\to S^2$ or $\pi\co  X \to D^2$, together with an $\omega$--compatible symplectic
form $\Omega$ is defined by: 
\begin{equation*} 
\area (X, \Omega)= \Vol (X, \Omega)/
\Vol (M, \omega)= \frac{\int _{X} \Omega ^{n+1}}{(n+1)\int _{M} \omega ^{n}}.
\end{equation*}
\end{definition}

The area
of $ \wwtilde{\Omega}_{\gamma}^{\infty}$ on $M \times D^2_\infty$ is
constructed to be $L ^{+} (\gamma) + \epsilon$.

By definition of $\sim$, $ (x, \theta)_0 \mapsto (\gamma _{\theta} (x),
\theta)_\infty$. Thus,  $$ \frac{\partial}{\partial \theta} \mapsto (\gamma
_{0})_* (X _{H_\gamma ^{\theta}} ) + \frac{\partial}{\partial \theta}, \quad
\frac{\partial}{\partial x} \mapsto (\gamma_\theta)_* \Big(
\frac{\partial}{\partial x}\Big) \quad \text {and} \quad \frac{\partial}{\partial r}
\mapsto - \frac{\partial}{\partial r}.$$
It follows that the gluing relation $\sim$ pulls
back the form $ \wwtilde{\Omega}_{\gamma}^{\infty}$ to the form 
\begin {equation*} \wwtilde{\Omega}_{\gamma}^{0}=   \omega  +\epsilon
\cdot 2rdr\wedge d\theta, \end {equation*} on the neighborhood of the boundary
$M \times \partial D^2_0$, which  extends to the form 
$\wwtilde{\Omega}_{\gamma}^{0} 
$ on $M\times D^2_0$ with area $\epsilon$. Then $ \{
\smash{\wwtilde{\Omega}_{ \gamma}\}}$ on $U$ is given by gluing
\begin {equation*} 
( \freels \times M \times D^2_0, \wwtilde{\Omega
}^{0}_\gamma) \cup (\freels  \times M \times D^2_\infty, \wwtilde{\Omega}
^{\infty}_\gamma)/ \sim.
\end {equation*}
The area of each fiber
is \begin{equation*} 
\area(X _{ \gamma}, \wwtilde{\Omega}_\gamma) =  L^{+} (\gamma) +
2\epsilon.
\end{equation*}
We pull back the family $\{
\wwtilde{\Omega}_{ \gamma}\}$ on ${U}$ to a family $\{
\smash{\wwtilde{\Omega}_{( \gamma, s)}}\}$ on $ U\times S^\infty$ via projection to
$ {U}$.
The $S^1$ action $ \tilde{\rho}$ does not act by a symplectomorphism from
the
fiber $X _{(\gamma,s)}$ to the fiber $X _{ (\tau \cdot \gamma,\tau \cdot 
s)}$. 
 We can fix this problem 
by averaging. Define a family $\smash{\{\wwtilde{\Omega} ^{S^1}
 _{ (\gamma,s)}\}}$ on $ U\times S^\infty$ by
\begin {equation*} 
\wwtilde{\Omega} ^{S^1}  _{ (\gamma,s)}  = \frac{1}{2\pi}\int
_{S^1}
 \tilde{\rho} (\tau) ^{*}\wwtilde{\Omega} _{(\theta \cdot \gamma, \theta
 \cdot s)} \,d\tau.
\end {equation*} On $(\freels \times S^\infty)\times M \times D ^{2} _{\infty} $ this
form is
\begin{multline*}
 \wwtilde{\Omega} ^{S^1}  _{ (\gamma,s)} = \omega  - \max_x {H_\theta^\gamma
 (x}) d \eta \wedge d\theta -  \epsilon \cdot 2rdr \wedge d\theta \\+
 \frac{1}{2\pi}\int _{S^1}\left(d  ( \eta (r) H^{ \gamma}_ \theta (\gamma (\tau) ^{-1}x))
 \wedge
d\theta \right) d\tau. 
\end{multline*}
 It follows that each $
\wwtilde{\Omega} ^{S^1} _{ (\gamma,s)}$ is symplectic and $$\area (X _{
(\gamma,s)},
\wwtilde{\Omega} ^{S^1}
_{\gamma,s}) = L ^{+} (\gamma) + 2\epsilon $$ as before.  
Thus, the family $ \smash{\wwtilde{\Omega} ^{S^1}  _{ (\gamma,s)}}$ on $U\times
S^\infty$ passes down to
a family $ \{\Omega _{ b}\}$  on the quotient bundle $p\co  U^{\smash{S^1}} \to Q$ with 
\begin{equation} \label {eq.ch1.area} \area \{\Omega _{ b}\} = L ^{+} (\gamma)
+ 2\epsilon.
\end{equation}

\subsection {Equivariant cycles in \texorpdfstring{$\freels$}{LHam}} \label {section.costr.eq.cycles} 
Let  $B$ be oriented compact and smooth. Up to homotopy, every cycle $f\co  B \to
Q$ arises as follows.
Let $g\co  Y \to
 B$ be a smooth principal $S^1$ bundle. And let $\hat{f}\co  Y \to
\Ham(M, \omega)$ be a map. Define 
\begin{equation*} 
o\co  Y \to LY
\end{equation*}
 to be the map which
sends
$x \in Y$ to the loop $\gamma_x$, $\gamma_x (\theta)=  x \cdot \theta$, also let  $
{f'}\co  LY \to \freels $ be the map induced by $\hat{f}\co Y \to
\Ham(M, \omega)$. Set $ \tilde{f}=f' \circ o$, then 
\begin{equation} \label {eq.widetildaf}\tilde{f}\co  Y \to \freels
\end{equation}
 is $S^1$
equivariant. Let $c\co  Y \to S^\infty$ be an $S^1$ equivariant map. (The $S^1$
equivariant homotopy class of this map is uniquely determined.) Consider the
product map
\begin{equation*} \tilde{f} \times c \co  Y \to \freels \times S^\infty.
\end{equation*} This is again an $S^1$ equivariant map under the diagonal $S^1$
action and so induces a map on 
the quotients  $f\co B \to Q$, whose homotopy class is independent of the choice
of the classifying map $c$. 
\begin{definition} \label {smooth} We will call $f\co  B \to Q$ \emph{
smooth}, if it comes from a smooth map $ \hat{f}\co  Y \to
\Ham(M, \omega)$. 
\end{definition}
Clearly any map $f\co  B \to Q$ can be perturbed to be smooth.
\begin{example}  \label{section.homotopy.groups} Let's apply the above
construction to a map $\hat{f}\co  S^ {2k+1} \to \textrm {Ham} (M, \omega)$. We
can associate to it two cycles in $Q$, by slicing
$S ^{2k+1}$ by
circles in two different ways. The first cycle, $f_h\co  \mathbb{CP} ^{k} \to Q$
is obtained from the Hopf
fibration $h\co  S ^{2k+1} \to \mathbb{CP} ^{k}$. 
The second $f\co  S ^{2k} \to Q$ is obtained from the trivial fibration
$\pr\co  S ^{2k} \times S^1 \to S^{2k}$ and the composition $ \hat{f}_2= \hat{f}
\circ t\co  S ^{2k} \times S^1 \to \Ham(M, \omega)$, where $$t\co S ^{2k}
\times S^1 \to S ^{2k+1}$$ is any fixed degree 1 map. The maps $f$ and
$f_h$
are not homologous since any such homology would project to a homology in $
\mathbb{CP} ^{\infty}$, for the classifying maps of the bundles $\pr\co  S
^{2k} \times S^1 \to S ^{2k}$ and $h\co  S ^{2k+1} \to \mathbb{CP} ^{k}$. 
\end{example} 
  
\begin{remark}  \label {remark.reformulation} Given a smooth map $f\co  B \to
Q$ the pullback bundle $p_f:
P_f
\to B$ by $f$ of the bundle $p\co  U^{\smash{S^1}} \to Q$ can be given the following tautological
reformulation, which will be useful to us. The map $f$ comes
from a smooth map $ \hat{f}\co  Y \to \Ham(M, \omega)$ for a certain
smooth oriented principal $S^1$ bundle $g\co  Y \to B$. This
induces a map $ \tilde{f}\co
Y \to \freels$, where $ \tilde{f}$ is as in \eqref{eq.widetildaf}.
Consider the pullback bundle 
\begin{equation} \label {bundle.tilda} p _{ \tilde{f}}: P _{ \tilde{f}}
\to Y
\end{equation}
by
$ \tilde{f}$ of the bundle $ \tilde{p}\co  U\to \freels$. In
other words 
\begin{equation*} P _{ \tilde{f}}= (Y \times M \times D ^{2}_0) \cup (Y \times M
\times D^2_\infty)/\sim
\end{equation*}
where $ (y, x, 1, \theta)_0$
 is equivalent to $ (y, {\tilde{f}}_{t,\theta} (x),1,\theta)_\infty$.
This
is a smooth bundle with the pullback of the $S^1$--action $ \tilde{\rho}$ on
${U}$ given by 
\begin{equation} \label {rho} \theta' \cdot (y,x,r,\theta)
_{0, \infty}= (\theta' \cdot y, x, r, \theta-\theta') _{0, \infty}.
\end{equation}
The quotient by the $S^1$ action on this bundle is the bundle $p _{f}\co P _{f}
\to B$. Thus, when $f\co  B \to Q$ is smooth the bundle $p _{f}\co P _{f} \to B$ and
the family $ \{f ^{*} (\Omega_b)\}$ of symplectic forms on this bundle are
smooth.
\end{remark}
\subsection {Natural embeddings into an \texorpdfstring{$\mathcal {F}$--}{F-}fibration} Now,
let $f\co  B \to Q$ be as usual, $ \tilde{f}\co Y \to \freels$  the associated $S
^{1}$--equivariant map (cf \eqref{eq.widetildaf}) and consider the
associated fibration $P _{ \tilde{f}}$ (cf \eqref{bundle.tilda}). There
are natural embeddings 
\begin{equation*} 
\tilde{I} _{0, \infty}: Y \times M
\to Y \times M \times D
^{2} _{0, \infty},
\end{equation*}
 given by including into the fiber over $0 \in D
^{2} _{0, \infty}$ 
and thus induced embeddings $$\tilde{I} _{0, \infty}\co  Y \times M \to P _{
\smash{\tilde{f}}}.$$ These
maps are $S^1$ equivariant under the action of $ \tilde{\rho}$ (cf \eqref{rho}) and hence there are  induced embeddings $I _{0, \infty}\co  B \times M \to P
_{f}$, which be used later. 

\subsubsection*{A special case} \label{special case} 
If we consider $Q$ as $\freels$ bundle over $ \mathbb{CP} ^{\infty}$, there is a
natural map $i_*\co  H_* (\freels) \to H ^{S^1}_*
(\freels)$ induced by inclusion of the fiber.
Given a cycle ${f}'\co  B \to \freels$, the bundle $P _{f}$ induced by 
the cycle $f=i \circ{f}'\co  B \to Q$ can be easily seen to be
isomorphic to the
pullback by $ \tilde{f}$ of the bundle ${U}$ over $\freels$, ie
\begin{equation*} P _{f} \simeq (B \times M \times D ^{2}_0) \cup  (B \times M \times D^2_\infty) /\sim
\end{equation*} 
where for $(b,x,\theta)_0$  in the boundary of $B \times M \times D ^{2}_0$,  $
(b,x,\theta)_0 \sim (b, {f}' _{b, \theta}(x), \theta)_\infty$, and the
embeddings \begin {equation} \label {eq.18}I_z\co  B \times M \to P_f,
\end {equation} defined above are now defined for all $z \in S^2$. (This
embedding is only well defined up to isotopy for $z$ in the equator $\partial
D^2 _{0, \infty} \in S^2$.) 

The following subsection essentially sets up for \fullref{section.lie.group} and its reading may be postponed until then. On the
other hand, it
may help to clarify the above constructions.  
\subsection {Example of an \texorpdfstring{$ \mathcal {F}$--}{F-}fibration} \label {Y.action} 
Suppose now we have a map $\hat{f}\co  Y \to \Ham(M, \omega),$
where $q\co  Y \to B$ is an oriented principal $S^1$ bundle. Suppose further that the map $
\hat{f}$ 
is $S^1$ equivariant with respect to the $S^1$ action on $Y$ and $S^1$ action
on $ \Ham(M, \omega)$ corresponding to the right action by a subgroup
$\gamma\co  S^1 \to \Ham(M, \omega)$ on $ \Ham(M, \omega)$.  Let us
understand the fibration $P _{f}$ for the induced map \begin{equation*} f\co  B \to Q. 
\end{equation*}
First, we can identify $X _{\gamma}$ with $S^3 \times _{S^1} M $, where
$S^1$ acts diagonally on $S ^{3} \times M$ by $$ e ^{2 \pi i \theta} \cdot (z_1,
z_2; x) = (e ^{-2 \pi i \theta} z_1, e ^{-2 \pi i \theta} z_2; \gamma (e
^{2 \pi i\theta})x),$$ using complex coordinates on $S^3$. To see
this, write $ [z_1, z_2; x]$ for the equivalence
class of the point $ (z_1/r, z_2/r; x) \in S ^{3} \times M$, where $r$ is the
norm of $(z_1, z_2)$. We identify
$ D_0 \times M$ with $ \{ [1,z;x]: |z| \leq 1, x \in M\}$ naturally and $D_\infty \times
M$ with $ \{[z, 1; x]: |z| \leq 1, x \in M\}$ via the orientation
reversing reflection. The gluing map is then
\begin{equation*} [1, e ^{2 \pi i \theta}; x] \sim [e ^{-2 \pi i \theta}, 1;
\gamma (e ^{2 \pi i \theta})x], \end{equation*} consistent with the
previous definition. There is an $S^1$ action $\beta$ on $X _{\gamma}$ given by
\begin{equation} \label {eq.action.beta} e^{2 \pi i\theta'} \cdot [z_1, z_2; x]
= [z_1, e ^{2 \pi i \theta'} z_2; x].
\end{equation}

\begin{lemma} \label {lemma.isomorphism} The bundle $p _{f}\co  P _{f} \to B$ is
isomorphic to the bundle $h\co  Y \times _{S^1} X _{\gamma} \to B$, where
$S^1$ is acting  by $\beta$ on $X _{\gamma}$. 
\end{lemma}
\begin{proof} 
Let $ \tilde{f}\co  Y \to \freels$
be as above (see \eqref{eq.widetildaf}) so \begin{gather*}\big(\tilde{f} (y)=
\hat{f}(y) \circ \gamma \big)\co  S^1 \to \Ham(M, \omega).\\ 
%
{P _{ \tilde{f}}}= Y  \times M \times D^2_0 \cup Y
\times M \times D^2_\infty / \sim 
\tag*{\hbox{We have}}\\
(y, x, 1, \theta)_0 \sim (y, \hat{f}(y) \circ
\gamma_\theta(x), 1, \theta)_\infty.
\tag*{\hbox{where}}
\end{gather*} 
%
In coordinates we have
\begin{equation*} 
M \times D ^{2}_0 \cup Y \times M \times D ^{2}_\infty /\sim
\end{equation*} where $ (y, x, 1, \theta)_0 \sim (y, \gamma_\theta (x), 1, 
\theta)_\infty$. There is a map $k\co  {P} _{ \tilde{f}} \to Y \times X
_{\gamma}$ defined as follows: 
\begin{align*} k(y,x, r, \theta)_\infty &= (y, \hat{f}(y) ^{-1} (x), r,
\theta)_\infty  \\
 k(y,x, r, \theta)_0 &= (y, x, r, \theta)_0. 
\end{align*}
This is a well defined bundle map, as is shown by the
following diagram: $$
\xymatrix{(y,x, \theta)_0   \ar [r]^-{\hbox{\footnotesize$\sim$}} \ar [d]^ -{k} & (y, \gamma
^{y}_\theta (x),\theta)_\infty \ar  [d]^- {k}\\
 (y,x, \theta)_0  \ar [r]^-{\hbox{\footnotesize$\sim$}} & (y, \gamma_ \theta(x)= \hat{f} (y) ^{-1}
 \circ \gamma ^{y}_{\theta}(x), \theta)_\infty}$$
We have the $S^1$ action on $Y
\times X _{\gamma}$ given by 
\begin {align} \label {eq.action.beta2}\theta' \cdot (y, x, r, \theta)_0 & = ( 
y \cdot \theta', x, r, \theta -\theta')_0\\
\label {eq.action.beta3}\theta' \cdot (y, x, r, \theta)_\infty & = (y \cdot
\theta', \gamma ^{-1} (\theta')x, r, \theta -\theta')_\infty. 
\end {align}
It is now not hard to check that the map $k$ is $S^1$ equivariant with respect
to the action $ \tilde{\rho}$ (cf \eqref{rho}) and the action given in
\eqref{eq.action.beta2}--\eqref{eq.action.beta3}.  Finally,
we conclude that \[ P _{f} \simeq Y \times _{S^1} X _{\gamma}.
\proved\]
\end {proof} 

\subsubsection {An admissible family of symplectic forms on \texorpdfstring{$P _{f}$}{P\137f}}
\label{admissible.family} Suppose $P _{f}$ is as in \fullref{lemma.isomorphism}. Using this lemma we can put an 
admissible family $ \{\Omega_{ {b}}\}$ on $p\co  P_f \to B$, $ b \in
B$ and
a compatible family of almost complex structures $\{J_ {b}\}$ as follows.
Let $\alpha$ be the standard contact form on the unit sphere $S^3$, normalized
so that $d\alpha=h^* \tau$, where $h\co  S^3 \to S^2$ is the Hopf map and  $\tau$
is a standard area form on $S^2$ with area 1. 
If $H\co  M \to \mathbb{R}$ denotes the normalized Hamiltonian generating $\gamma$, the
closed $2$--form \begin{equation*}  \omega- (\max H+ \epsilon) d\alpha + d (H\alpha)
\end{equation*} on $S^3 \times M$ descends to a form $ \tilde{\omega}$
on $S ^{3} \times _{S^1} M$, which is symplectic for an $\epsilon>0$. 
Let $J$ be any $S^1$--invariant almost complex structure on $M$ and $J_0$ the
standard $S^1$ invariant complex structure on
$ \mathbb{C} ^{2}$. Then $J \times J_0$ is also $S^1$--invariant, and its
restriction to $S^3$ preserves the contact planes $\ker \alpha$. It is not hard
to see that $J \times J_0$ descends to an almost complex structure $
\tilde{J}$ on the quotient $X _{\gamma}$ which coincides with $J$ on the fibers $M$.
By construction, if $J$ is compatible
with $\omega$, then $ \tilde{J}$ is compatible with $ \tilde{\omega}$.
The form $ \tilde{\omega}$ and the complex structure $ \tilde{J}$ are
invariant under the
$S^1$ action $\beta$ on $X _{\gamma}$ and therefore give rise to a family $
\{\Omega_ {b}\}$ and a compatible family $ \{J _{b}\}$ on $P_f = Y \times _{S^1}
X _{\gamma} $. 
\subsection {PGW invariants of an \texorpdfstring{$ \mathcal {F}$--}{F-}fibration} \label{Section2}
Let $p\co  P \to B$ be an $ \mathcal {F}$--fibration. 
\begin{definition} We call a family $\{J_b\}$ of fiberwise
$\{\Omega_b\}$--compatible complex  structures \emph{$\pi$--compatible}
if  $\pi\co  (X _{b}, J_b) \to (S ^{2}, j)$  is holomorphic for each $b$ and each
$J_b$ preserves the $\Omega_b$--orthogonal subspaces of $TX_b$, for some
admissible family $ \{\Omega _{b}\}$ on~$P$.
\end{definition}

Let $ \{J_b\}$ be a
$\pi$--compatible family of almost complex structures. Consider the following moduli space 
 $$ \mathcal M^*_ {0}(P, {A}, \{J_b\}) =\{ \text {pairs }(u,b)\},$$ where
\begin{itemize}
   \item $b \in B ^{k}$.
   \item $u$ is a $J_b$--holomorphic, simple  curve $u\co  (S^2, j) \to X_b
   \subset P$  representing class $A \in j_* (H_2 ^{\sect}(X))
   \subset H_2 (P)$, where $H_2 ^{\sect} (X)$ are section classes and $j_*$ is
   induced by inclusion of fiber. (The subspace  $j_* (H_2 ^{\sect}(X))$ is
   unambiguous, since the structure group $ \mathcal {F}$ preserves
   section classes of $X$.)
   \end{itemize}
An element of the above moduli space will be called loosely a \emph{fiber 
holomorphic curve}. 
For details of the following discussion see for example 
McDuff and Salamon~\cite[Sections 6.7, 8.4]{MS} or Buse~\cite{OB}.
Given an element $ (u, b)$ of the moduli space $\mathcal M^*_ {0}(P, {A},
\{J_b\})$ there is the associated real linear Cauchy--Riemann operator
\begin{equation*} 
D_{u, b}: \big\{\xi \in \Omega ^{0} 
(S^2, u
^{*}TP) \mid dp (\xi) \equiv \text {const} \big\} \to \Omega ^{0,1} (S^2, u
^{*} TX_b)
\end{equation*}  
of index $2n+k +2 c_1 (A)$, where $c_1$ is the vertical Chern class of the
fibration $p\co  P \to B$. 
A $\pi$--compatible family  $ \{J_b\}$ is called
\emph{regular} for $A$, if the operator $D _{u,b}$ is surjective for every
tuple $ (u,b)$, where $u \in \mathcal {M} ^{*}_0 (P, A, \{J_b\})$.
 The set of regular $\pi$--compatible families for $A$ will be
denoted by $ \mathcal {J} _{\reg}  (A)$ and the set of all families by $
\mathcal {J}$. From now on regular family $ \{J_b\}$ always refers to a
$\pi$--compatible regular family.
\begin{lemma}
\begin{enumerate} 
\item If $ \{J_b\} \in \mathcal {J} _{\reg} (A)$ then  $
\mathcal M^*(P, {A}; \{J_b\})$ is a smooth manifold of dimension 
\begin{equation*} \dim \mathcal {M}^*(P,{A}; \{J_b\}) = 2n +k + 2c_1 (A).
\end{equation*}
\item The set $ \mathcal {J} _{\reg} (A)$ is of the second category in $ \mathcal {J}$.
\end{enumerate}
\end{lemma}
Suppose now we have an oriented smooth cobordism $C$ between $B_1, B_2$. Let $P
_{C}$ be a symplectic fibration over $C$. We denote by $P _{i}$ the
restriction of $P _{C}$ over $B_i$. Suppose we have regular
families $\{J_b ^{i}\}$ on $P _{i}$.
Let $ \{J ^{C}_b\}$ be family on $P _{C}$ restricting to $\{J_b ^{i}\}$ on $P
_{i}$. We then have  the corresponding moduli space
\begin{equation*} \mathcal {M} ^{*} (P_C, A; \{J ^{C}_b\}).
\end{equation*} 
We again say that $ \{J ^{C}_b\}$ is \emph{regular} if  the associated
Cauchy--Riemann operator is surjective. The space of regular families $\smash{ \{J ^{C}_b\}}$
will be denoted by $ \smash{\mathcal {J} _{ \text {reg}} (A;  \{J_b ^{1}\}, \{J _{b} ^{2}\})}$,
and the space of all families by $ \smash{\mathcal {J} (A; \{J_b ^{1}\}, \{J _{b} ^{2}\})}$. 
\begin{lemma} \label{cobordism}
\begin{enumerate}
\item If $\{J ^{C}_b\}$ is regular $\mathcal {M} ^{*} (P_C, A; \{J ^{C} _{b}\})$
is a smooth oriented manifold with boundary 
\begin{equation*} \partial \mathcal {M} ^{*} (P_C, A; \{J_b ^{C}\}) = 
\mathcal {M} (P_2, A, \{J_b ^{1}\}) - \mathcal {M} ^{*} (P_1, A, \{J_b
^{2}\}).
\end{equation*}

\item The set $ \mathcal {J} _{ \text {reg}} (A;  \{J_b ^{1}\}, \{J _{b} ^{2}\})$ is of
the second category in  $ \mathcal {J} (A; \{J_b ^{1}\}, \{J _{b} ^{2}\})$.
\end{enumerate}
\end{lemma} 

$\hbox{Let}\hfill  \mathcal M^*_ {0,l}(P, {A}, \{J_b\})= \big\{ \text {equivalence classes
of tuples } (u,z_1 \ldots z_l) 
\big\}, \hfill\hphantom{Let}$ 

where $u \in \mathcal {M} ^{*} (P, A; \{J_b\})$  and $z_1, \ldots,
z_l$ are pairwise distinct points in $S^2$. The equivalence relation is 
$ (u, z_1, .., z_l) \sim (u', z _{1}', \ldots , z _{l}')$ if there exists $\phi
\in PSL (2, \mathbb{C})$ s.t.\ $u' \circ \phi=u$ and $\phi (z_i)= z_i'$.

 For a
regular family $ \{J_b\}$, this is a manifold of dimension 
\begin{equation} \label {dimension} 2n+k+2c_1 (A) +2l -6, 
\end{equation}
where $k$ is the dimension of the base $B$.
Consider the evaluation map $$\ev\co \mathcal M^*_ {0,l}(P, {A}; \{J_b\})
\to P^l.$$ Similarly, we have evaluation maps $\ev ^{C}\co \mathcal {M} ^{*} (P_C,
A; \{J ^{C} _{b}\}) \to P _{C}$.
For these maps to represent 
pseudocycles we need some conditions on $M$. 
\begin{proposition} \label {proposition.pseudocycle} Let $ (M, \omega)$ be
spherically monotone. Then the maps $\ev$ and $\ev ^{C}$ above are pseudocycles
for generic regular $\pi$--compatible families $ \{J_b\}$ and $ \{J_b ^{C}\}$.
\end{proposition}
\begin{proof} 
Since we only consider curves which lie in the fibers of $p\co  P \to B$, any
bubbles must lie in the fiber. Next note that a stable map
into $(X _{b}, J _{b})$, representing a section class of $\pi\co  X_b \to S^2$ must
consist of a principal part which is a section, together with ``bubbles''
which lie in the fibers $M$ of $\pi\co  X _{b} \to S ^{2}$; see 
McDuff~\cite[Lemma 2.9]{Dusa}. By assumption that $M$ is monotone, these bubbles must have
positive Chern number. Using this, one can show that for a generic $\pi$--compatible family the evaluation map is a pseudocycle by standard arguments in McDuff and Salamon~\cite[Chapter 6]{MS}.
\end{proof}
\subsection {Definition of PGW invariants}
Under the assumptions of \fullref{proposition.pseudocycle}, we define
parametric Gromov--Witten invariants by
 \begin{equation*} \PGW ^{P}_{0,l} (a_1, \ldots, a_l; A)= [\ev] \cdot
(a_1 \times \ldots \times a_l),
\end{equation*} where $\cdot$ denotes intersection pairing in $P ^{l}$ and
$a_1, \ldots, a_l \in H_* (P)$.
\subsection {Quantum homology} 
The flavor of quantum homology we use is the following.
Let $\Lambda: = \Lambda ^{ \textrm {univ}} [q, q ^{-1}]$  be the ring of Laurent
polynomials in a variable $q$ of degree 2 with coefficients in the universal
Novikov ring.  Thus, its elements are polynomials in $q$ of the form 
\begin {equation} \label {eq.lambda} \sum_{\epsilon \in \mathbb{R}, \quad l \in \mathbb{Z}}
\lambda _{\epsilon, l} \cdot q ^{l} t ^{\epsilon} \qquad \#
\{\lambda _{\epsilon, l} \neq 0|  \epsilon \geq c \} <  \infty \textrm { for
all } c \in \mathbb{R},
\end {equation}
where $ \lambda _{\epsilon, l} \in \mathbb{Q}$. Set $$  {QH}_* (M)=
 {QH}_* (M ; \Lambda)=
 {H}_* (M) \otimes _{ \mathbb{Z}} \Lambda.$$
We define a valuation $\nu\co  QH_* (M) \to \mathbb{R}$ as follows:
\begin{equation*} \nu \Big(\sum_{A} b_A \cdot  q ^{l_A} t ^{\epsilon_A}\Big):= \sup _{
b_A \neq 0}  \epsilon_A,
\end{equation*}
where $A$ is an abstract index. 

Recall that the quantum intersection product on $QH_* (M)$ is
defined as follows. 
For $a, b \in H_* (M)$ $$a \ast _{M}b= \sum _{A \in H_2 ^{S} (M)} (a \ast _{M}b)_A
\otimes q ^{-c_1 (A)} t ^{-\omega (A)},$$ where 
$ (a \ast_M b)_A \in H _{i+j-2n+2c_1 (A)} (M)$ is defined by the duality
\begin{equation*} (a \ast_M b)_A \cdot c = GW _{0,3} ^{M} (a,b,c; A), \qquad
\text {for all $c\in H_* (M)$}.
\end{equation*}
The product is then extended by linearity to all of $QH_* (M)$. This product can
be shown to be associative (see McDuff and Salamon~\cite [Chapter 11, Section 1]{MS} for details)
and gives $QH_* (M)$ the structure of a graded commutative ring with unit $[M]$.
  
\section {Definition of QC classes} 
\label{quantum.char.classes} Let $X$
 as before be a Hamiltonian fibration:  $\pi\co X \to S^2$ with monotone fiber $M$,
and $p\co  P \to B$ be a smooth $ \mathcal {F}$--fibration with fiber $X$,
classified by a map into $\freels$, cf \fullref{remark.structure.group}. 
The following is a important ingredient in the definition of QC classes and
plays the role of the 2 dimensional cohomology class of the curvature form in
Chern--Weyl theory.
Let $M _{ \Ham}$ denote the universal $M$ bundle over $\BHam(M, \omega)$. There is a unique class $ [\Omega]\in H ^2 (M _{ \Ham})$ called the coupling class such that $$i^* ( [\Omega])= [\omega],
\hspace {15pt} \int_M {[\Omega]}^ {n+1}=0 \in H^2 (\BHam(M, \omega))$$  where $i\co  M \to M _{
\Ham} $ is the inclusion of fiber, and the integral above denotes the
integration along the fiber   (see K{\c{e}}dra and McDuff~\cite[Section3]{kedra-2005-9}). 
Note from \eqref{eq8} that the total space $P$ of the bundle $p\co  P \to B$ has
another submersive projection to $B \times S^2$ and the resulting bundle $M
\hookrightarrow P \to B \times S^2$ is clearly Hamiltonian, ie the transition
maps are fiberwise Hamiltonian symplectomorphisms.
\begin{definition}
We denote by $\mathcal {C} \in H ^{2} (P)$, the pullback
of the class $ [\Omega]$ above, by the classifying map of the natural
Hamiltonian fibration $$M \hookrightarrow P \to B \times S^2.$$
\end{definition}

Set $QH_* ^{B} (M)= H_* (B \times M) \otimes \Lambda.$
\begin{definition}  We define the \emph{total quantum characteristic
class}  of $p\co  P \to B$ by \begin{equation*} 
c ^{q} (P)= \sum_ {A \in {j_*(H_2^{\sect}  (X ))}} b _{{A}} \otimes q^ {-c_{ \textrm {vert}} ( {A})} t ^{
-\mathcal {C} ( {A})}  \in QH_* ^{B} (M). \end{equation*} In this formula,\begin{itemize}
   \item  $H ^{\sect}_2 (X)$ denotes the section homology classes of $\pi\co  X \to
   S^2$ as in \mbox{\fullref{Section2}}.     
    \item The map $j_*\co  H ^{\sect}_2(X) \to H_2(P_f)$ is
 induced by inclusion of fiber. 
  \item The coefficient $b_A \in H_* (B \times
M)$ is the transverse intersection of $$\ev: {\mathcal M}_ {0,1}(P, A; \{J_b\}) 
\to P$$ with $I _{0} (B \times M)$ (see \eqref{eq.18}).  More formally, 
$b _{A}$ is defined by  duality $$b _{A} \cdot _{B \times M} c = [\ev] \cdot
_P {I_0}_* (c),$$ for $c \in H_* (B \times M)$.
 \end{itemize}
\end{definition}
The above definition works essentially without change for an $ \mathcal
{F}$--fibration classified by a map into $Q$.
 \begin{remark}  To deduce that the condition
\eqref{eq.lambda} on the coefficients is satisfied we need to show that  there are
only finitely many homology classes $A \in H _{2}(P)$ which have
 representatives with area less then $c$ for every $c>0$ (for a fixed
 Riemannian metric on the compact manifold $P$). For then in
 particular
 there are only finitely many homology classes represented by vertical $ \{J_b\}$--holomorphic curves with $$E(A)=\Omega _{b} (A)
 =\mathcal {C} (A)+ \pi ^{*}
 (\alpha_b)  (A) \leq c$$ for every $c>0$, which would imply the finiteness
 condition. To prove the former intuitive statement, one can use geometric
measure theory and compactness theorem for spaces of integral currents with
uniformly bounded mass norm (cf Federer~\cite{GMT}).
\end{remark}
 \begin{notation} Let
us from now on shorten notation by setting $$q^
{-c_{ \textrm {vert}} ( {A})} t ^{ -\mathcal {C} ( {A})} \equiv e ^{A},$$
where it presents no confusion. 
\end{notation}
For a regular family $\{J_b\}$, $ \mathcal M^*_ {0,1}(P, {A}, \{J_b\})$
is a smooth manifold of dimension  $$(2n+2) + m + 2c_{1}(A) -4 = (2n +2) +m +2c
_{ \text {vert}} (A),$$ where $m=\dim B$ (cf \eqref{dimension}). It
follows that 
\begin {equation} \label {degree} \deg b_A= 2n+m +2c _{\text{vert}}
(A).
\end {equation} 
In particular, a class $A$ contributes only if $2 c _{\mathrm{vert}} (A) \leq 0$.

Every element $e = \sum_ {{j_*(A)}} b _{{A}} \otimes e ^{A} \in QH_* ^{B} (M)$
defines a linear functional on
$H _{*} (B)$ (where $H _{k} (B)= H _{k} (B, \mathbb{Z})/ \text { Tor}$) with
values in $QH_* (M)$ defined as follows. 
If $a \in H_k (B)$, then  $e (a) \in
QH_* (M)$ is given by
\begin {equation} \label {eq.functional}e (a)  =  \sum_ {A} \sum _{i} \left(b _{{A}}
\cdot  (a \otimes e _{i} ^{*}) \right) e_i \;  \otimes e ^{A},
\end {equation}
where $\{e_i\}$ is a basis for $H_* (M)$, $\{e_i ^{*}\}$ a dual basis for $H
_{*} (M)$ with respect to the intersection pairing and $\cdot$ is the
intersection pairing on $H_* (B \times M)$.  
\begin{remark}  Let us check  the degree of $e (a)$. We have that $
b _{{A}} \cdot  (a \otimes e _{i} ^{*})$ is nonzero when $$2n+m+2c _{\vert} (A) + \deg
a+2n- \deg e_i= 2n + m$$ so we get $\deg e _{i}=2n+\deg a+ 2c _{\vert} (A)$
and 
\begin {equation} \label {degree2} \deg e (a)= \deg e_i - 2 c _{\vert}
(A)= 2n+\deg a.
\end {equation} 
\end{remark} 
\begin{definition}  We  define the \emph{$k$--th quantum
characteristic class} $$c
^{q} _{k} (P) \in \text {Hom} (H_k (B), QH _{*} (M)) =H ^{k} (B,
QH_* (M)),$$ to be the restriction of the functional $c ^{q} (P)$ to $H _{k}
(B)$. 
\end{definition}
  In these terms, the functional $c ^{q} (P)$ is just the sum
$$c ^{q} (P)= c ^{q} _{0} (P) + c ^{q} _{1} (P) + \ldots +c ^{q} _{m} (P),$$
where $m$ is the dimension of $B$. When 
$\gamma$ is contractible, \fullref{Axiom 1} implies that $c
^{q}_0 (P) (\pt)= [M]$ is the multiplicative identity in the quantum
homology ring.
The analogous expression for Chern classes is called the total Chern class.
Interestingly, in our ``quantum'' setting the total class has a nice geometric
interpretation and this plays  a role in proving   the
corresponding ``Whitney sum formula'' in \fullref{section5}. 
\begin{example}  A loop $\gamma\co  S ^{1} \to \Ham(M, \omega)$ can be
viewed as a map $f_\gamma\co  \pt \to \freels$. The corresponding fibration $P
_{f_\gamma}$ over $\pt$ has fiber $X _{\gamma}$ and  $$c ^{q} (P _{f _{\gamma}})=
\sum_ {{j_*(A)}} b _{{A}} \otimes q^ {-c_{ \textrm {vert}}  ( {A})} t ^{
-\tilde{\omega} ( {A})} \in QH _{2n} (M),$$ since $H_* (\pt \times M)
\simeq H_* (M)$ and  \eqref{degree2} implies that the degree of the element $c ^{q} (P _{f _{\gamma}})$ is $2n$. 
In these terms, the Seidel element corresponding to $\gamma$ is defined to be 
\begin {equation*} S
(\gamma) = c ^{q} (P _{f_\gamma}). 
\end {equation*}
This element depends only on the homotopy class of $\gamma$, and Seidel
\cite{Seidel} proved 
that this defines a homomorphism $S\co  \pi_1( \Ham(M, \omega)) \to QH _{2n}
(M)$.
\end{example}
Recall from the introduction that for a smooth $k$--cycle $f\co  B \to Q$, $$\Psi
(B, f) \equiv c ^{q} _{k} (P _{f}) ( [B]).$$ 

\begin{lemma} \label{independence} The characteristic classes $c ^{q}_k (P)$ of
$p\co  P \to B$ are
independent of the choice of the admissible family $ \{\Omega_b\}$, and moreover
$\Psi (B_1,f_1) = \Psi (B_2, f_2)$ if $f_1\co  B_1 \to Q$ 
 is oriented cobordant to $f_2\co  B_2 \to Q$, in particular $\Psi$ is well defined
 map on $H_* (\ls, \mathbb{Q})$.
\end{lemma}
\begin{proof} 
 To prove that $c ^{q} _{k} (P)$
are independent of the choice of the admissible family $ \{\Omega_b\}$  note
that by \fullref{structures} any two such families are smoothly homotopy
equivalent. The homotopy $ \{\Omega_b ^{t}\}$ gives an admissible family of
forms on $p\co  P \times I \to I$ at which point we may apply \fullref{cobordism} and \fullref{proposition.pseudocycle}. 
  
To prove the second statement 
consider a smooth oriented cobordism $F\co  C \to Q$ between  
$ (B_1,f_1)$ and $ (B_2, f_2)$. The proof is just a simple
consequence of \fullref{cobordism}. 
The construction
in \fullref{setup} yields an $ \mathcal {F}$--fibration $P_F$ over $C$
restricting
to the $ \mathcal {F}$--fibrations $P_i$ over $B_{i}$. Moreover, for $c \in H_*
(M)$ the class $I_0 ([ C] \otimes c)$ in $P _{F}$ restricts to the corresponding
classes $I_0 ([B_{i}] \otimes c)$ in $P_i$, cf \eqref{eq.18}. 
Let $$\Psi (B _{i}, f_i)= \sum _{A} b ^{i}_A  \otimes e ^{A}$$ be the
corresponding elements in $QH_* (M)$. 
We need to show that $b_A ^{1}=b ^{2}_A$.  
Consider the intersection numbers
 \begin{align*} b ^{i}_A \cdot _{M} c \equiv [\ev ^{i}_A] \cdot _{P} I 
 _0 ([B_i] \otimes c),
\end{align*}
where $\ev_A ^{i}$ are the evaluation maps 
\begin{align*} \ev_A ^{i}\co  \mathcal {M} _{0,1} (P _{i}, A, \{J ^{i}_b\}) \to
P _{i}
\end{align*}
for regular families $ \{J ^{i}_b\}$. 
Let $$\ev ^{F}_{A}\co  \mathcal {M} _{0,1} (P _{F}, A, \{J ^{C}_b\}) \to
P _{F} $$ be the evaluation map with $ \{J ^{C}_b\}$ a regular family
restricting to $ \{J ^{i}_b\}$ on $P _{i}$. When  the
dimension of $c$ is such that the intersection numbers above are nonzero, 
\fullref{cobordism}, and \fullref{proposition.pseudocycle} imply
that $\ev ^{F}_A \cap I  ([C] \otimes c)$ is a one-dimensional cobordism between the oriented $0$--dimensional manifolds 
$\ev ^{1}_A \cap I _0 (B_1 {\otimes} c)$, $\ev ^{2}_A \cap I _0
( B_2 {\otimes} c)$, assuming things are perturbed to be transverse. Thus, the
intersection numbers $\ev ^{1}_A \cdot I _0 (B_1 \otimes c)$, $ \ev ^{2}_A
\cdot I _0 (B _{2} \otimes c)$ coincide. 

To conclude that $\Psi$ is well
defined on $H_* (\ls, \mathbb{Q})$ we may use \fullref{thm.milnor.moore}, which implies that 
the rational 
 homology of $\ls$ is generated
 by cycles $f\co  B \to X$, where $B$ is a closed oriented smooth
 $k$--manifold (in fact a product of spheres). 
Moreover, \fullref{thm.milnor.moore} implies that the relations in the
rational homology of $\ls$ are generated by maps of smooth cobordisms (actually
cylindrical cobordisms).
\end{proof}
\section {Verification of axioms} \label{axiom1and2} 
\begin{proposition} \label {proposition.pullback} Let $p\co  P \to B$ be an $ \mathcal
{F}$--fibration and $f\co  C \to B$ a smooth $k$--cycle
representing $a \in H_* (B)$. Then $$c_k ^{q} (P) (a)= \Psi (f ^{*} P) \equiv
c ^{q}_k (f ^{*} P ( [C])).$$
\end{proposition}
\begin{proof} 
Let $ \{J_b\}$ be a regular family for $A$ curves in $P$.
We have maps
\begin {equation*} \xymatrix{{\mathcal {M}} ^{*} _{0,1} (P, A;
\{J_b\}) \ar [r]^-{\ev}  & P \ar [d]^-p \\
C \ar [r]^-f & B} 
\end {equation*}
Perturb $f\co  C \to B$ to be transverse to the pseudocycle $$p \circ \ev\co {\mathcal
{M}} _{0,1} ^{*} (P, A,
\{J_b\}) \to B,$$  and consider the
commutative diagram
\begin{equation*}\xymatrix{Z  \ar [r]^- {\pr_1} \ar [d]^-{\pr_2} &  \mathcal {M}
_{0,1} ^{*} (P, A, \{J_b\}) \ar [d]^-{\ev}  \\ f ^{*} (P) \ar [d] \ar [r] & P \ar
[d]^-p \\ C \ar [r]^-f & B}
\end{equation*}
where $Z$ is the pullback of the diagram. By the transversality above $Z$ is a
smooth manifold and can be tautologically identified with $\mathcal {M} ^{*} _{0,1}
( P', A,  \{J'_b\})$, where $p'\co  P' \to C$ is the pullback bundle $f ^{*} (P)$ over $C$ and $ \{J'_b\}= \{ f ^{*}
(J_b)\}$. Moreover, the evaluation map $\ev'\co  
\smash{\mathcal {M} ^{*}_{0,1}} ( P', A,  \smash{\{J'_b\}}) \to P'$ is just the map $\pr_2$ and is
a pseudocycle since $\ev$ is a
pseudocycle. The dimension of this pseudocycle is $$\dim [\ev] + k -m= (2n+
m +2c_1 (A)-2)+k-m= 2n+k+2c_1 (A)-2,$$  where $m$ is the dimension of $B$ and
 the expected
dimension of $ \mathcal {M} ^{*}_{0,1} ( P', A,  \{J'_b\})$. Thus, $\ev\co
\mathcal {M} _{0,1} ^{*}( P', A,  \{J'_b\}) \to P'$ is a pseudocycle of the
correct dimension. 

We show that the family $ \{J'_b\}$ is regular. The linearized Cauchy-Riemann
operator for $b$ in the intersection of $p \circ \ev$ with $f$ has the form:
\begin{equation*} D_{u, b}: \Omega ^{0} _{B} \equiv \big\{\xi \in \Omega ^{0} 
(S^2, u
^{*}TP) | p_* (\xi) \equiv \text {const} \big\} \to \Omega ^{0,1} (S^2, u
^{*} TX_b).
\end{equation*} 
By the regularity assumption
on $\{J_b\}$ 
this operator is onto. Moreover, by regularity we have
\begin{equation*} p_* (\ker
D_{u,b})= p_* \circ \ev_* (T_u \mathcal {M}_{0,1} (P, {J_b})).
\end{equation*}
Thus, by the transversality assumption we must have that 
\begin{equation} \label {eq.eq.p_*} p_*\co  \ker D_{u,b} \to
T_b B/ {f_* (TC)|_b} 
\end{equation}
is onto. Denote by $D _{u,b} ^{C}$ the restriction of the
operator $D _{u, b}$ to the subspace $$\Omega ^{0}_C \equiv \big\{\xi \in
\Omega ^{0} (S^2, u^{*}TP | p_* (\xi) \equiv \text {const} \in f_* {(TC)|_b})
\big\}.$$ To show that $ \{J'_b\}$ is regular we must show that $D _{u,b} ^{C}$ is also
onto. Let $ \tilde{v} \in \Omega ^{0}_B$. By \eqref{eq.eq.p_*} there
exists  $\tilde{v}_k \in \ker D _{u,b}$ and $v_C \in f_* (TC)|_b$, s.t. 
\begin{equation*} p_* ( \tilde{v})= p_* ( \tilde{v}_k) + v_C. 
\end{equation*} 
Therefore, we get that $\tilde{v} - \tilde{v}_k \in \Omega ^{0}_C$,
and so $D _{u,b} (\tilde v) = D _{u,b} ^{C} (\tilde{v} -
\tilde{v}_k)$. Since $D _{u,b}$ is onto, it follows that $\smash{D _{u,b} ^{C}}$ is
also onto and so $ \smash{\{J'_b\}}$ is regular.

 By definition, \begin{align*}c_k ^{q} (P) (a) &= \sum_ {{A}} \sum _{i}
b _{{A}} \cdot  (a \otimes e _{i} ^{*}) e_i \otimes e ^{A},\\
c_k ^{q} (P') ( [C])&= \sum_ {A} \sum _{i} b' _{{A}}
\cdot  ( [C] \otimes e _{i} ^{*}) e_i \otimes e ^{A}.
\end{align*}  
To finish the proof, we 
note that by the above discussion 
\[
\eqalignbot{ b_A \cdot _{B \times M} (a \otimes e_i) & \equiv [\ev] \cdot _P {I
_{0}} _{*} (a \otimes e_i) \cr & = [\ev'] \cdot _{P'} {I _{0}} _{*}( [C] \otimes
e_i) \cr & \equiv [b'_A] \cdot _ {C \times M} ( [C] \otimes e_i).
}\proved\]
\end {proof}
\subsection[Verification of \ref{Axiom 1}]{Verification of \fullref{Axiom 1}}
 To prove the first statement just apply \fullref{proposition.pullback} to $i\co  \pt \to
 B$.
To prove the second statement note that if $P \simeq X \times B$ then we can
take a constant family of regular compatible almost complex structures $ \{J
^{\reg}\}$ and this family is clearly parametrically regular. It follows that
the total characteristic class is \begin{equation*} c^q (P)= \sum_ {A} (B \otimes b'_A) \otimes e ^{A},
\end{equation*} where $b'_A$ is the transverse
intersection of $\ev: \mathcal M_ {0,1}(X, {A}, J ^{reg}) 
\to X$ with the fiber $M_0 \subset X$ over $0$. As a functional on $H_* (B)$, 
$c ^{q}  (P) (a) =0$ unless $\deg (a)=0$. 
\hfill\qedsymbol

\subsection[Verification of \ref{functoriality}]{Verification of \fullref{functoriality}} If $f\co  C \to B$ represents $a
\in H_k (B_1)$ as before, then 
\begin{equation*} g ^{*} c ^{q}_k (P_2) (f_* [C])= c ^{q} _{k} (P_2) (g_*f_*
[C]) = \Psi ( f ^{*} g ^{*} P_2), 
\end{equation*} 
where the last equality holds by \fullref{proposition.pullback}, and
\begin{equation*} c_k ^{q} ( g ^{*} P_2) (f_* [C])= \Psi ( f ^{*} g ^{*} P_2), 
\end{equation*} again by \fullref{proposition.pullback}.

\subsection[Proof of \ref{main theorem} assuming \ref{thm.whitney.sum}]{Proof of \fullref{main theorem} assuming \fullref{thm.whitney.sum}}
\begin{definition} The Pontryagin product $$f_1 \star f_2\co  B_1 \times
B_2 \to \ls$$  of two maps
$f_1, f_2\co  B_1,B_2 \to \ls$ is defined by $$f_1 \star f_2 (b_1,b_2,
\theta)= f_2 (b_2, \theta) \circ f_1 (b_1, \theta).$$ (The order is indeed reversed.)
 
\end{definition}
\begin{proof} For $i=1,2$, let $f_i\co  B_i \to \freels$ be as before. Let $k_i$ be the dimension
of $B_i$. Consider the maps $$ \tilde{f}_i\co B_1
\times B_2 \to
\ls, \qquad \tilde{f}_i=f_i \circ \pr_i \quad \text{ for } i=1,2,$$ where
$\pr_i \co  B_1 \times B_2 \to B_i$ are the component projections. 
Clearly, $$P _{f_1 \star f_2} \simeq P _{ \tilde{f}_1} \oplus P _{
\tilde{f}_2} \equiv P _{ \widetilde{f_2} \cdot \widetilde{f_1}}$$ (see
\fullref{definition.sum}).  By \fullref{functoriality} and \fullref{thm.whitney.sum}, 
\[
\eqalignbot{ \Psi (B_1 \times B_2, f_1 \star f_2) & \equiv c ^{q} _{k_1+k_2} (P
_{f_1 \star f_2}) (B_1 \times B_2) \cr
 &= \sum _{i+j=k_1+k_2} \pr_1 ^{*}(c
^{q}_{i} (P _{f_1})) \cup \pr_2 ^{*} (c ^{q}_{j} (P _{f_2}))
(B_1 \times B_2) \cr
& = 
\pr_1 ^{*} (c
^{q}_{k_1} (P _{f_1})) \cup \pr_2 ^{*} (c ^{q}_{k_2} (P _{f_2})) (B_1 \times B_2) \cr
& = c
^{q}_{k_1} (P _{f_1}) (B_1) \ast c ^{q}_{k_2} (P _{f_2}) (B_2) \cr
& = \Psi (B_1, f_1) \ast \Psi (B_2, f_2).
}\proved
\]\end{proof}

\begin{remark}  \label {remark.seidelrep} Under the Pontryagin product, 
the group ring of $ \pi _{1} (\Ham(M, \omega)) $ over $ \mathbb{Q}$ is $H_0
(\ls, \mathbb{Q})$. 
The restriction of $\Psi$ to degree zero, $$\Psi ^{0}\co 
{H}_0 (\ls, \mathbb{Q}) \to QH _{2n} (M),$$  is a ring homomorphism
$$S=\Psi^0\co  \mathbb{Q}[\pi_1 ( \Ham(M,
\omega))] \to QH _{2n} (M),$$ in view of \eqref{eqS}. Thus, \fullref{main theorem} is an extension of the Seidel homomorphism $S$ to the entire
Pontryagin ring $H_* (\ls, \mathbb{Q})$.
\end{remark}
\subsection[Verification of \ref{thm.whitney.sum}]{Verification of \fullref{thm.whitney.sum}} \label{section5} 
In this subsection we prove that the classes $c ^{q} _{k}$
satisfy \fullref{thm.whitney.sum}. To this end we will need a splitting
formula for PGW invariants arising from the connected sum operation  on two $
\mathcal {F}$--fibrations. To help clarify the picture we first explain why $P_1
\oplus P_2$ is the connected sum of $P _{1}, P _{2}$ in an appropriate way.
 \begin{definition} Let $P _{1}, P
_{2}$ be
two $ \mathcal {F}$--fibrations classified by $f ^{1}, f ^{2}\co  B \to \freels$. 
Define 
\begin{equation*} P _{1} \# P _{2} \equiv (B \times M \times D ^{2}_0) \cup (B
\times M \times S^1 \times I) \cup (B \times M \times D ^{2} _\infty)/ \sim 
\end{equation*}
where the equivalence relation is
\begin{align*} (b, x, 1, \theta)_0 &\sim (b, f ^{1}
_{b, \theta} (x), \theta, 0) \in B \times M \times S ^{1} \times I\\
(b, x, 1, \theta)_\infty &\sim (b, (f ^{2}_{b,\theta}) ^{-1} (x), \theta, 1) \in B
\times M \times S ^{1} \times I.
\end{align*}
\end{definition}
It is then not hard to construct a natural isomorphism between $P_1 \oplus
P_2$ and $P_1 \# P _{2}$. 
Given classes $A$ and $B$ in $j_* H ^{\sect}_2 (X_1) \subset H_2(P _{1})$
 respectively $j_* H ^{\sect}_2 (X_2) \subset H_2(P _{2})$,  there is a natural
 section class $A \# B$ in $H_2(P _{1} \# P _{2})$.  To define this class
 one represents $A$ and $B$ by sections coinciding in the  fiber over $\infty$
 for $X ^{1}_b$ and  the fiber over $0$ for $\smash{X ^{2}_b}$ respectively, (this can
 be made more precise using the definition above). It can  be directly checked
 that the class $A \# B$ is independent of choices and  satisfies $$c _{\vert}
 (A \#B)= c _{\vert}(A) + c _{\vert} (B) \text  { and }\mathcal {C} _{P_1 \#P_2}
 (A \#B) = \mathcal {C} _{P_1} (A) + \mathcal  {C} _{P _{2}} (B).$$ Given $a
 \in H_* (B \times M)$, we will denote the class      $(I ^{P}_z)_* (a) \in H_*
 (P)$ by $a$ for shorthand and similarly for  $P_1, P_2$. \begin{theorem}
 \label{splitting formula} Let  $P _{1}, P _{2}$ be two $ \mathcal {F}$--fibrations classified by $f ^{1}, f ^{2}\co  B \to \freels$,
$P=P_1 \oplus P_2$ their connected sum, $A \in H_2 (P_1), B \in H_2 (P_2)$ 
as in \fullref{def.quant.class}.
Then for all $a_1, \ldots, a_k \in H_*
(B \times M)$, and any integer $0 \leq l \leq k$,
\begin{multline*}
\PGW^{P} _{0,k} (a_1, \ldots, a_k ;C)\\=
\smash{\sum_ {\substack{i \\ {A_1}\# B_1= {C}}}} \PGW_{0,l+1}
^{ P_2} (a_1,\ldots, a_l, e_i; {A_1}
 ) \\ \cdot \PGW^ {P_1} _{0,k-l+1} (e_i ^{*}, a_ {l+1}, \ldots, a_k; {A_2}),
\end{multline*}
where $ \{e_i\}$ is a basis for $H _{*} (B \times M)$, $\{e^*_i\}$
is the dual basis.
\end{theorem} 
\begin{proof}  
Set $K=B \times M$.
Suppose we have two $J$--holomorphic curves $u_1, u_2$ into $P _{{f^1}}$
and $P
_{ {f^2}}$ in class $A _{1}, A _{2}$ intersecting in $K$, 
(where we identify $K$ with its
embedding in $\smash{P _{\smash{{f}^1}}}`$ by $I_\infty$ and in $\smash{P _{\smash{{f}^2}}}`$ by $I
_{0}$), then $u_1, u_2$ lie in the respective fibers $\smash{X ^{1} _{b}}$ and $\smash{X
^{2}_b}$ over the same point $b \in B$. We can then glue them to get a curve in
the fiber $X^1_b \# X ^{2}_b \simeq X _{f ^{2} (b)\,\cdot f ^{1} (b)}$ over $b$ of
the fibration $P _{1} \oplus P_2$ by exactly same argument as in Section
11.4 of \cite{MS}. 

One then shows that for generic
families $ \{J ^{1}_b\}$, $\{J ^{2}_b\}$ the moduli spaces $$ \mathcal {M} ^{*}
(P _{ {f} ^{1}}, A_1;  \{J ^{1}_b\}) \quad \text {and} \quad \mathcal {M} ^{*} ( P _{
{f} ^{2}}, A_2;  \{J ^{2}_b\})$$ are regular and the evaluation map 
\begin{equation*}  \mathcal {M} ^{*} (P _{
{f} ^{1}}, A_1;  \{J ^{1}_b\}) \times \mathcal {M} ^{*} ( P _{
{f} ^{2}}, A_2;  \{J ^{2}_b\}) \to K \times K 
\end{equation*} which takes $ (u ^{1}, u ^{2}) \mapsto (u ^{1}_\infty (0), u ^{2}_0 (0)
)$ is transverse to the diagonal. The rest of the proof is exactly the same as the
proof of the corresponding splitting statement in Chapter 10 of \cite{MS}. 
\end{proof}
\begin{remark}  Note that since all holomorphic curves of $P_1, P_2, P_1
\# P_2$ come from section classes (of the fiber $X^1, X^2, X^1 \# X^2$
respectively) they are necessarily transverse to the divisor $K$ and intersect
it in a single point. This formula is then ``essentially'' a special case of
the formula  given by Ionel and Parker~\cite{IP} for general symplectic sums
along a codimension 2  submanifold; see also Li and Ruan~\cite{Li.Ruan} for a different approach.
The main difference in our case is that we don't have global symplectic forms
on $P_1, P_2, P_1 \# P _{2}$ but rather families of forms. 
\end{remark} 
In what follows we think of $H ^{*} (B,
QH_* (M))$ as the space of linear functionals on $H_* (B)$ with
values in $QH_* (M)$. In particular an element in $H ^{*} (B,
QH_* (M))$ can be of mixed degree.
Thus, by the Kunneth formula and Poincare duality $H ^{*} (B,
QH_* (M))$ is naturally identified with $QH_* ^{B} (M)$ via
\eqref{eq.functional}. To avoid confusion for an element $a \in QH_*
^{B} (M)$ we will write $\PD (a)$ when we think of it as an element of $H ^{*} (B,
QH_* (M))$.

Considering the following elements $a, a' \in H_* (B \times M) \subset QH_* ^{B}
(M)$ \begin{gather*} a= \sum _{i} a_i \otimes m_i, \quad a'=
\sum _{j} a'_j \otimes m_j', \quad a_i, a_i' \in H_* (B), m_i, m_j' \in H_* (M),
\\
\PD(a) \cup \PD(a')= \PD \Big (\sum _{i, j} (a_i \cap a'_j)
\otimes m_i \ast m_j' \Big),
\tag*{\hbox{then}}
\end{gather*} 
where $\ast$ is the homology
quantum product. We will need the following simple Lemma. 
\begin{lemma} \label{lemma.cup.product} If $P _{\tr}=B \times (M \times S^2)$
and $a, a'$ as above, then 
\begin{equation*} 
\begin {split} \PD(a) \cup \PD(a') &= \PD \Big (\sum
_{A} \big(\PD (a) \cup \PD (a') \big)_A  e ^{A } \in QH_* ^{B} (M) \Big),\\
(\PD (a) \cup \PD (a'))_A &= \sum
_{k, l} \PGW ^{P _{\tr}} _{0,3} (a, a', e _{k,l}; A) e ^{*}_{k,l},
\end {split}
\end{equation*}
where $\{e
_{k,l}=b_k \otimes e_l \}$ is a basis for $H_* (B \times M)$.
\end{lemma}
\begin{proof}  Let $ \{J
^{ \text {reg}}\} $ be the constant family of regular
complex structures on $P _{\tr}$ compatible with a constant admissible family $
\{\Omega\}$. Then the  family $ \{J ^{ \text {reg}}\}$ is itself parametrically 
regular. We have $$\sum
_{k,l} \PGW ^{P _{\tr}} _{0,3} (a, a', e _{k,l}; A) e
^{*} _{k,l} = \sum
_{k,l} \sum _{i,j} \PGW ^{P _{\tr}} _{0,3} (a_i \otimes m_i, a_j' \otimes m_j',
e _{k,l}; A) e ^{*} _{k,l}.$$ As oriented manifolds, 
\begin {equation*} 
\mathcal {M} _{0,3} ^{*} (P _{\tr},
A; J ^{ \text {reg}}) \simeq B \times \mathcal {M} _{0,3} ^{*}(\tr, A; \{J ^{ \text
{reg}}\}).   
\end {equation*}
 Moreover the diagram
$$ \xymatrix{ B \times \mathcal {M} ^{*} _{0,3} (M \times S^2, A, J ^{ \text
{reg}}) \ar [r]^-{\ev ^{\tr}} \ar [d] & P _{\tr} ^{3} \simeq B ^{3} \times (M \times S^2)
^{3} \ar [d] \\
B \ar [r]^- {\mathrm{diag}} & B ^{3}}$$ commutes. Hence, $[\ev ^{{\tr}}]= 
[\mathrm{diag}] \otimes [\ev ^{M \times S^2}] $ as a cycle in $B ^{3} \times (M \times
S^2) ^{3}$ with the orientation pulled back from the orientation on $P_{\tr}
^{3}$ via the identification $\smash{P_{\tr} ^{3} \simeq B ^{3} \times (M \times
S^2)^3}$, where $\smash{\ev ^{\tr}}$ and $\smash{\ev ^{M \times S^2}}$ are the evaluation maps from 
\linebreak$\smash{\mathcal {M} ^{*} _{0,3} (P_{\tr}, A, \{J ^{ \text {reg}}\})}
$ and $\smash{\mathcal {M} ^{*} _{0,3} (M \times S^2, A, J ^{ \text {reg}})}$, respectively. 

Therefore, 
\begin{equation*}
\begin{split} 
 \sum
_{k,l} \PGW ^{P _{\tr}} _{0,3} &(a_i  \otimes m_i,  a'_j \otimes m'_j, e _{k,l};
A) e ^{*} _{k,l}   \\
 &= \sum _{k,l} \left ([\ev ^{\tr}] \cdot ( a_i \otimes a'_j
\otimes b_k) \otimes (m_i \otimes m'_j \otimes e_l) \right ) (b_{k} \otimes
e_l) ^{*} \\ & =
 \sum
_{k,l}  \,[\mathrm{diag}] \otimes[ \ev ^{M \times S^2}]   \cdot ( a_i \otimes a'_j \otimes
b_k) \otimes (m_i \otimes m'_j \otimes e_l) (b _{k}  \otimes e_l) ^{*}  \\
    & =\sum _{k}((a_i \cap a'_j)
\cdot b_k) b_k ^{*} \otimes \sum _{l} GW ^{M \times S^2} _{0,3} (m_i, m_j, e_l;
A) e_l ^{*}
 \\ & = (a_i \cap a'_j)
\otimes \sum _{l} GW ^{M \times S^2} _{0,3} (m_i, m_j, e_l;
A) e_l ^{*},
\end{split} 
\end{equation*}
where we used that $ \smash{[\mathrm{diag}] \cdot a_i \otimes a'_j \otimes
b_k = (a_i \cap a'_j)
\cdot b_k}$. 
Summing over all $A \in H_2 ^{\sect} (M \times S^2)$ we get the desired equality.
\end{proof}
Given an $ \mathcal {F}$--fibration
$P$, define $m ^{P}\co  H_* (B \times M) \to QH_* ^{B} (M)$
 by 
\begin {equation} \label {eq.2.2.5}m ^{P} (a) = \sum _{A,i} \PGW ^{P} _{0,2} (a,
e_i; A) e_i ^{*} \otimes e ^{A},
\end {equation} where $e_i$ are as in \fullref{splitting formula}, and extend by linearity to
all of $QH_* ^{B} (M)$. 
\begin{lemma} If $P=P_1 \oplus P_2$ then 
$$m ^{P}= m ^{P_1} \circ m ^{P_2}.$$ 
\end{lemma}
\begin{proof} By \eqref{eq.2.2.5},
\begin{equation*}
\begin{split} 
m^{P_1} \circ m^{P_2} (a) &= \sum _{C} \sum _{i, A\#B=C} \PGW ^{P_2} _{0,2} 
(a, e_i; A) \cdot \PGW ^{P_1} _{0,2} (e _{i} ^{*}, e _{j}; B) e ^{*} _{j} \otimes
e ^{C} \\
& = \sum _{i, C} \PGW ^{P} _{0,2} (a, e_j,C) e ^{*}_j\otimes e
^{C} = m^{P} (a),
\end{split}
\end{equation*}
where we used \fullref{splitting formula} for the second equality.
\end{proof}
\begin{lemma}  For an $ \mathcal {F}$--fibration $P$, 
$$\PD(m ^{P} (a))= \PD(c ^{q} (P)) \cup \PD(a).$$ 
\end{lemma}
\begin{proof} It suffices to prove this for a simple class $a \in QH_* ^{B} (M)$. Using \fullref{splitting formula} with $P
_{2}=P$ and $P _{1}= \text {tr}=B \times (M \times S^2)$ we get:
\begin {equation*}
\begin{split}   \PD(m ^{P} (a)) &= \PD \Big (\sum _{j,C} \PGW ^{P} _{0,2} (a, e_j;
C) e ^{*}_j \otimes e ^{C} \Big)\\ 
& =\PD \Big( \sum _{i,j, A\#B=C} \PGW ^{P} _{0,1} (e_i;A) \cdot \PGW ^{ P_{\tr}}
_{0,3} (e _{i} ^{*}, a, e_j; B) e ^{*}_j \otimes e ^{C} \Big )\\
& = \sum _{i, A \# B= C} \PGW _{0,1} ^{P} (e_i; A) \cdot \big(\PD(e ^{*}_i)
\cup \PD(a)\big)_{B} \otimes e ^{C} \\
& = \PD(c^q (P)) \cup \PD(a),
\end{split}
\end {equation*} where we used \fullref{lemma.cup.product}  for the next to
last equality.
\end{proof}
\subsection[Verification of \ref{thm.whitney.sum}]{Verification of \fullref{thm.whitney.sum}} Using the above lemmas we get,
\begin{gather*}
\hfill \PD(c ^{q} (P))= \PD(m ^{P} (B \times M)) = \rlap{$\PD(m ^{P_1}
\circ m ^{P_2} (B \times M))$} \hphantom{\PD(c ^{q} (P_1)) \cup \PD(c
^{q} (P_2)).}  \hfill \\ 
\hfill\hphantom{\sq \PD(c ^{q} (P))= \PD(m ^{P} (B \times M))}
= \PD(c ^{q} (P_1)) \cup \PD(c ^{q} (P_2)).\hfill \sq\\
\end{gather*}

\section {QC classes and the Hofer geometry} \label {QC.and.Hofer}
Let $p\co  P _{f} \to B$ be a smooth $ \mathcal {F}$--fibration. We explain here how
$c^q (P_f)$ gives rise to lower bounds for the positive max-length measure $L^+(f)=\max_
 {b\in B} L ^{+} (f_b)$; this will be used later in this section. We will
 assume that the family $ \{\Omega_b\}$ on $P _{f}$ has been chosen so that
 condition  \eqref{eq.ch1.area} is satisfied. Let $f\co  B \to Q$ be a general
 smooth cycle. Define a valuation \begin{equation*} \nu\co QH_* (M), QH_* ^{B} (M)
 \to \mathbb{R} \text    {\quad by \quad}\nu \Big(\sum_{A} b_A \cdot t
 ^{\epsilon_A} q ^{l_A}\Big):= \sup _{ b_A \neq 0} \epsilon_A,
\end{equation*}
and $b_A$ is in $H_*(M)$ or $H_*(B
\times M)$.
Our next
proposition is a direct generalization of Seidel's (see McDuff~\cite{DM2}). 

\begin{proposition} \label {lower bound} We have
\begin{equation} \label {eq.nu1} \nu(c^q(P _{f}))
\leq \min _{{(B, f)} \in [H]} \Big(\max _{b \in B} L ^{+} (\gamma_b)\Big),
\end{equation}
where $[H]$ represents the homotopy class of maps $f\co  B \to Q$ and
$\gamma_b$ is the loop $f (b)$ (this is defined up to an action of $S^1$).
Moreover, 
\begin{equation} \label {eq.nu2} \nu (\Psi (B,f)) \leq \min _{{(B, f)} \in [H]}
\Big(\max _{b \in B} L ^{+} (\gamma_b)\Big),
\end{equation} 
where $[H]$ now denotes the bordism class of maps $f\co  B \to Q$.
\end{proposition}

\begin{proof} \label {proof.lower.bound}
Let $$c^q (P)=\sum_ {{A}} b _{{A}} \otimes q^ {-c _{ \text {vert}} ( {A})}
 t ^{ -\mathcal {C} ( {A})}.$$  If $b_A \neq 0$ in $H _{*} (B \times M)$
 then there is a $J_b$--holomorphic
 curve $u\co  (S^2, j) \to  X _{b} \subset P_f$ in class $A \in \smash{H_2 ^{S} (P_f)}$.
On the other hand $[\Omega_b] = \mathcal {C}+\pi^* ([\alpha_b])$, for
some area form $\alpha$ on $S^2$, where $ [\Omega _{b}]$ is the
cohomology class of $\Omega _{b}$ in $H ^{2} (X _{b})$. Since $\Omega_b$ tames
$J_b$, we get $$0< [\Omega_b] (A)=  \left ( \mathcal {C}+ \pi^* (\alpha_b) \right ) (A)=
\mathcal {C} ( A)+ \area( p ^{-1} (b), \Omega_b).$$ Therefore,
\begin{equation} \label {eq.last} -\mathcal
{C}  ({A}) <\area( p ^{-1} (b), \Omega_b)= L^{+} (\gamma_b) + 2\epsilon \leq
\max _{b \in B} L ^{+} (\gamma_b) + 2\epsilon
\end{equation} 
for all $A$. Passing to the
limit in $A$ and $\epsilon$ we get $\nu ( c^q (P _{f})) \leq \max _{b \in B} L ^{+} (\gamma_b) $.  Since the left 
hand side of \eqref{eq.last} depends only on the homotopy class
of $f$, we get the inequality \eqref{eq.nu1}. Inequality \eqref{eq.nu2} follows
by the same argument and \fullref{independence}. 
\end{proof}  
\subsection {Calculation for some symmetric \texorpdfstring{$ \mathcal {F}$--}{F-}fibrations}
\label {section.lie.group}

Let $q\co  Y \to B$ be a principal $S^1$--bundle  and
$\hat{f}\co  Y \to \Ham(M, \omega)$ an $S^1$--equivariant map with respect to the right action of
some subgroup $\gamma\co  S^1 \to \Ham(M, \omega)$ on $ \Ham(M,
\omega)$. Recall from \fullref{Y.action} that we have an induced cycle $f\co B
\to Q$, and an induced fibration $p\co  P _{f} \to B$. 
In this section, we
give an expression for the ``leading-order term'' contribution to the total
quantum characteristic class $c^q (P _{f})$ and give a proof of \fullref{main.lie.group}. This extends the
calculation for $S^1$ actions in \cite [Theorem 1.10]{MT}.

By \fullref{lemma.isomorphism}, 
$P _{f}$ can be identified with $h\co  Y \times _{S^1} X
_{\gamma} \to B$. The bundle $Y \times _{S^1} X _{\gamma}$
comes with an admissible family $ \{\Omega_{b}\}$ and a compatible family $\{J_{b}\}$
constructed in \fullref{admissible.family}.
 To understand the
behavior of fiber holomorphic curves in $P _{f}$, we need to first understand $
\tilde{ J}$ holomorphic curves in $X _{\gamma}$, where $ \tilde{J}$ is
the almost complex structure described in \fullref{admissible.family}. 
Each fixed point $x$ of the $S^1$--action $\gamma$ gives rise to
a $ \tilde{J}$--holomorphic section of $X _{\gamma}$ defined by $$\sigma_x = S ^{3}
\times _{S^1}
\{x\} \subset X _{\gamma}.$$ Denote by $F_{\max}$ the maximal fixed point set
of the
Hamiltonian $S^1$--action $\gamma$ on $M$, ie the maximal set of the
generating Hamiltonian $H$ of $\gamma$. Let $\sigma_{\max} \in H_2 (X_\gamma)$
denote the homology class of the section $\sigma_x$ for $x \in F _{\max}$. For each $x \in F_{\max}$
we have a $ \tilde{J}$ holomorphic $\sigma _{\max}$--class curve.
An important observation due to Seidel is that these are the only $
\tilde{J}$--holomorphic curves in that homology class (cf
\cite [Lemma 3.1]{MT}); and so the moduli space of these unparametrized curves
is identified with $F
_{\max}$. Since the $S^1$--action $\beta$ (see
\eqref{eq.action.beta}) on $X _{\gamma}$  maps each section $\sigma
_{x}$ to itself, it follows
that the unparametrized moduli space $\mathcal M ^*
_{0,0}({P _{f}}, \sigma_{\max}; \{J_{b} \})$ can
be identified with $B \times F_{\max}$. In particular it is a compact manifold.
Let $E$ be the obstruction bundle over this moduli space. The fiber
of $E$ at $ (b, x) \in B \times F_{\max} $ is the cokernel of the operator
\begin{equation*} D_{u, b}\co  \big\{\xi \in \Omega ^{0}  (S^2, u
^{*}TP_f) | dp_f (\xi) \equiv \text {const} \big\} \to \Omega ^{0,1} 
(S^2, u ^{*}TX _{b}), 
\end{equation*}   
where $u\co  S^2 \to X_b$ parametrizes the
section $\sigma_x \subset X_b$. 
We write $\smash{D ^{\vert} _{u,b}}$ for the restriction of $D _{u,b}$ to $\Omega ^{0}  (S^2, u
^{*}TX _{b} )$. 
\begin{lemma} $\coker D ^{\vert} _{u, b} \simeq \coker D
_{u,b}$.
\end{lemma}
\begin{proof} 
Since the map 
\begin{equation*} p_f \circ \ev\co  \mathcal {M} ^{*}_{0,1} (P_{f}, \sigma
_{\max}; \{J_b\}) \to B
\end{equation*}
is a submersion, the homomorphism
\begin{equation*} dp_f\co  \ker D_{u,b} \to T_b B  
\end{equation*}
is onto. It easily follows that $D ^{\vert} _{u,b}$ and $D _{u,b}$ have the same
image.
\end{proof} 

Thus, the fiber $E_b$ of the obstruction bundle $E$ is $\smash{{\coker D ^{\vert}
_{u,b}}}$. The fundamental class of $\mathcal {M} _{0,0} ^{*} (P _{f}, \sigma
_{\max}; \{J_b ^{\reg}\})$ is identified with $\PD _{B \times F _{\max}} e (E),$
\; see \cite [Chapter 7.2]{MS}.  
We thus have the following direct generalization of \cite
[Theorem 1.9]{MT}.
 \begin{proposition}
\label{thm.leading.term}
Let $f\co  B \to Q$ and the obstruction bundle $E$ be as above. 
Then \begin{multline*}c^q (P_f)=  \PD _{B \times F _{\max}} e (E) \otimes q ^{-m _{\max}}t ^{H
_{\max}}\\ + \sum _{A \in H_2 ^{S} (M)| \omega (A)>0} b _{\sigma_{\max+A}} \otimes
q ^{-m _{\text{max}}- c_1 (A)} \; t
^{H_{\text{max}} -\omega ({A})},\end{multline*} where $m_{\max}=c _{\vert} (\sigma_x)=\sum
_{i} k_i$ and $H _{\max}$ is the maximum value of the normalized Hamiltonian 
generating $\gamma$.
\end{proposition}

\begin{proof} Since the evaluation map
$\ev\co  
\mathcal {M} _{0,1} ^{*} (P _{f}, \sigma _{\max}, \{J _{b}\}) \to P _{f}$
intersects $I_{0}
(B \times M)$ transversally at $B \times F _{\max}$, it can be readily deduced
from the above discussion that the class corresponding to the transverse intersection of
$$\ev ^{\reg}\co 
\mathcal {M} _{0,1} ^{*} (P _{f}, \sigma _{\max}, \{J _{\reg}\}) \to P_f$$ with $I_{0}
(B \times M)$ is $\PD _{B \times F} e (E) \in H_* (B \times M)$.

By \cite [Lemma 3.1]{MT} there are no contributions from sections $\sigma
_{\max} + A$ with $\omega (A) < 0$; this also follows
from the argument in the proof of \fullref{4}.
\end{proof}
To understand the obstruction bundle $E$, we need to understand cokernel of the
linearized Cauchy-Riemann operator
$$D ^{\vert}_{u, b}\co  \Omega ^{0}  (S^2, u ^{*}TX _{b}) \to  \Omega ^{0,1}
(S^2, u ^{*}TX_{ b}),$$ where $u: S^2
\to X _{ b}$ parametrizes the section $\sigma_x \in X _{ b}$, $x
\in F_{\max}$ as before.

 The
complex normal bundle $N({\sigma_x})$ of $\sigma _{x}$ inside $TX _{b}$ can be identified
with the bundle \begin {equation*} (T _{x}M, J _{x}) \times _{S^1} S ^{3} \to S^2
\end {equation*}
and so splits into a sum of complex line bundles 
\begin{equation} \label {eq.splitting} \bigoplus _{i=1} ^{n} L _{k_i},
\end{equation}
 where
the degree of $L_{k_i}$ is $k_i$. In other words each $S^1$ invariant summand $V_i \simeq \mathbb{C}$ of
$T _{x} M$, on which $S^1$ is acting by $v \mapsto e ^{-2\pi ik_i \theta} v$,
gives rise to the summand $L _{k_i}$ of $N (\sigma_x)$. Thus, $$TX
_{ b}| _{\sigma_x}= \Big(\bigoplus _{i=1} ^{n} L _{k_i}\Big) \oplus L _{2}\equiv L ,$$
where $L _{2}$ is the tangent bundle to $\sigma _{x}$. Since $x \in F
_{\max}$, $k_i \leq 0$ for all $z$. 

By proof
of \cite[Lemma 3.2]{MT} the operator $\smash{D ^{\vert}_{u_x, b}}$ 
is complex linear and is the Dolbeault operator $
\bar{\partial}$ on $\smash{TX _{\gamma}|_{ \sigma_x}}$, with respect to a holomorphic
structure for which the splitting \eqref{eq.splitting} is holomorphic. Thus,
the cokernel of $\smash{D ^{\vert}_{u_x, b}}$ is $\smash{H ^{0,1}_ {\bar \partial} ( S^2, L )}  \simeq \smash{(H
^{1,0}_ { \bar \partial} (S^2, L ^{*} )) ^{*}}$. The latter can be identified 
with $\smash{(H ^{0} (S^2, L^{*} \otimes K _{x})) ^{`*}}$, 
where $\smash{K _{x}=T^* (\sigma_x)}$ denotes the canonical bundle of $\sigma_x$. 
$$E _{ b, x,i}= H ^{0} (S^2, L _{k_i}^{*} \otimes K
_{x}).\leqno{\hbox{Set}}$$  
This latter space can be identified with the space of degree $n_i
\equiv -k_i-2$ homogeneous polynomials in
$X, Y$, where $X,Y$ denote the homogeneous coordinates on $ \mathbb{CP} ^{1}$.
Thus, a section in $\smash{H ^{0} (S^2, L
_{\smash{k_i}}^{*}} \otimes K
_{x})$ is completely determined by its holomorphic $n_i$--jet over $0  \in \smash{D^2_0}
\subset S^2$. Therefore, 
\begin{equation*} E _{ b, x,i}\simeq \mskip-2mu\bigoplus _{0 \leq j \leq
n_i}\mskip-3mu \text {Hom} \left((T_0 \sigma_x )^{\otimes j}, K_x|_0
\otimes L ^{*} _{k_i}|_0 \right) \mskip-2mu\simeq \mskip-2mu\bigoplus _{0 \leq j \leq
n_i} (K _{x}|_0 ^{\otimes j}) \otimes (K_x|_0
\otimes L ^{*} _{k_i}|_0) 
\end{equation*}
The cokernel $E _{ b, x}$  of $D ^{\vert}_{u_x, b}$ is  then
\begin{equation*} E _{ b, x}=\bigoplus _{i} E _{b, x,i} ^{*}, 
\end{equation*}
whose real dimension is the \emph{virtual index} of $\gamma$, defined by
 \begin{equation*} I(\gamma)=
 \sum _{\substack {1 \leq i \leq n \\ k_i  \leq -1}} 2(-k_i -1).
\end{equation*}  
 Let $ \wwtilde{ \mathcal {K}}$ be the bundle $Y 
\times_{S^1} \mathbb{C}$
and set $ \mathcal {K}=\pr_1 ^{*} \wwtilde{ \mathcal {K}}$, 
 where $\pr_1\co  B \times F_{\max} \to B$, and $\pr_2\co B \times F _{\max} \to F _{\max}$ are the
projections.

Then $ \mathcal {K}$ is the bundle over $B \times F
_{\max}$ whose fiber over $(x,b)= T _{0} \sigma _{x} =K_x |_0 ^{*}$, where $K_x=
T^* \sigma_x \subset X_{b}$ (cf \eqref{eq.action.beta}, \eqref{eq.action.beta2}). We also
have natural bundles $L _{i}$ over $B \times F _{\max}$ coming from the bundles
$L _{k_i}$ above.

Note that $e
(L_j)$
and $e ( \mathcal {K})$ are algebraically independent in the cohomology ring of
$B \times M$.
 The Euler class of $E$ is given by 
\begin{equation} \label {eq.31}
\begin{split} e (E) &= \prod_i \prod _{0 \leq j \leq n_i} ((j+1) e ( \mathcal {K}) + e
(L_i)).\\
& = \prod_i (n_i+1)! \;e ^{ \sum _{i} (n_i+1)} ( \mathcal {K}) + \text { mixed
terms}.
\end{split} 
\end{equation}
We can thus rewrite \eqref{eq.31}, using that $n_i = -k_i -2$, as
\begin{equation*}  e (E)= \sum _{0 \leq p \leq \frac{I(\gamma)} {2}} e
^{p} ( \mathcal {K})  \cup a _{p},
\end{equation*}
where $a _{p}$ are in $H ^{ \frac{I (\gamma)}{2} - p}(B \times F _{\max})$,
consisting of sums of products of classes $e (L_i)$ with some coefficients.
\begin{example} \label {ex.S3} Let $ \hat{f}\co  S ^{3} \to \Ham(
\mathbb{CP} ^{n}, \omega)$, $\gamma\co  S^{1} \to \Ham( \mathbb{CP} ^{n},
\omega)$ and the associated map 
$$f_h\co  S ^{2} \to Q= (\freels( \mathbb{CP} ^{n}, \omega) \times S
^{\infty})/ S ^{1}$$ be as in \fullref{section.example.S3}. Then $F
_{\max}= \max =[1, 0, \ldots, 0]$, $\smash{\mathcal {M} ^{*}_{0,1} (P_{f_h}, \sigma
_{\max}; \{J_b\})}$ is identified with $S ^{2}$ and the obstruction bundle $E$
is identified with the complex line bundle associated to the Hopf bundle $h\co  S
^{3} \to S ^{2}$, whose homological Euler class is $ [-\pt] \in H_0(S^2)$. Thus,
by \fullref{thm.leading.term} 
\begin{equation*} c ^{q} (P _{f_h})= ([-\pt] \otimes [\pt]) \otimes q ^ {- m
_{\max}} t ^{H _{\max}} + \text {lower $t$--order terms}.
\end{equation*} 
\end{example}
\begin{theorem} \label{thm.lie.group} Let $\hat {f}\co  Y \to \Ham( M,
\omega)$ be as above and $B= Y/S^1$. 
Every nonzero term $$e ^{p} ( \mathcal {K}) \cup a _{p} \in  H^*(B \times F _{\max})$$
in the expansion for $ e (E)$
 gives rise to a nontrivial characteristic class
$ c ^{q} _{2p} (P_f)$.
Moreover, it gives rise to cycles $f\co  C \to Q$,
minimizing the positive max-length measure in their bordism class.
\end{theorem} 
\begin{proof} If $e$ is the Euler class of $q\co  Y \to B$, then since $
\mathcal {K}$ is isomorphic to $\pr_1 ^{*} (Y_\gamma \times _{S^1} \mathbb{C})$  it
follows  that the Poincare dual of 
$ e ^{p} ( \mathcal {K}) \cup a _{p}$ is of the form $$\PD _{B}(e) \otimes
\PD _{F_ {\max}}(a_p| _{ [\pt] \times F_{\max}}) \in H_* (B \times M),$$ where
$\PD(e) \in H_* (B)$ and $\smash{(a_p| _{ [\pt] \times
F_{\max}}) ^{*}}$ is thought of as a class in $H_* (M)$ via inclusion of
$F_{\max}$ 
into $M$. Since the generating function $H$ of $\gamma$ is necessarily a perfect
Morse--Bott function (see McDuff and Salamon~\cite{MS2}) the inclusion of $F _{\max}$ into $M$ can
be shown to be injective on homology. The first part of the theorem is then
immediate from our assumption,
the definition of the characteristic classes $ c ^{q}_k (P_f)$, and \fullref{thm.leading.term}.  
We prove the second statement.  For some
$a \in H _{2p}  (B)$ we have that  $0 \neq
c ^{q}_{2p} (P _{f}) (a)$. By \fullref{proposition.pullback} $$c
^{q}_{2p} (P _{f}) (a) = \Psi (f \circ g, C),$$ where $g\co  C \to B$ is a smooth
map representing the rational homology class of $a$. Thus, the cycle $f \circ
g\co  C \to Q$ is essential in the bordism group by \fullref{independence}. Let
us see that it minimizes the max-length measure. By \fullref{thm.leading.term},  
$\nu ( \Psi (f
\circ g, C))= H _{\max}$. On the other $L^+(f \circ g)=H
_{\max}$, since all the loops in the image $ \text {Im} (f) \subset Q$ have
positive Hofer length $H _{\max}$. By \fullref{lower bound}
$f \circ g$ minimizes the measure $L^+(f)$ in its bordism class.  
\end{proof} 
\subsection[Proof of \ref{main.lie.group}]{Proof of \fullref{main.lie.group}} \label 
{section.ProofofTheorem.main.lie.group} Since $ e ( \wwtilde{\mathcal {K}}) =
e \neq 0$, the $p= \frac{I _{\gamma}}{2}$ term in the expansion of $e (E)$  is
nonzero. By \fullref{thm.lie.group}, the cycle  $f\co  B \to Q$ is 
essential and minimizes the measure $L^+(f)$ in its bordism class.
\qed 

\subsection[Proof of \ref{4}]{Proof of \fullref{4}} \label {proof.ref4}
Consider the fibration $h\co  S ^{2k+1} \to
\mathbb{CP}^k$. Homotop $\hat{f}\co  S^ {2k+1} \to \Ham(M, \omega)$,
so that  it takes the set $h ^{-1}( D^{c})$ to id, where $D \subset
\mathbb{CP}^k$ is an open ball. The new map will still be denoted by $\smash{\hat {f}}$.
Let $q\co  \mathbb{CP} ^{k} \to S ^{2k}$ be the quotient
map, squashing $ \mathbb{CP} ^{k} - D$ to $s_0 \in S ^{2k}$. There is an
induced quotient
map \begin{gather*}q| _{B}\times \id\co  (h ^{-1} (\bar {D}) \simeq \bar {D} \times S^1) \to S
^{2k}\times S^1.
\\
\hat {f} \big((q \times
\id)^{-1} (s_0 \times S ^{1}) \big)= \hat{f} ( h ^{-1} (\partial
\bar{D}))=\id, 
\tag*{\hbox{Since}}
\end{gather*}there is then an induced map $$ \tilde{f}\co 
S ^{2k} \times S^1 \to \Ham(M, \omega)$$ and the associated map
\begin{equation*} f_2\co   S ^{2k} \to \freels.
\end{equation*}
We will show
now that $ c ^{q} (P _{f_h})= c ^{q} (P _{f_2})$. On the other hand, we show in
\fullref{lemma3} below that $f$ is homotopy equivalent to $f_2$.

The restriction of $P _{f_h}$ to $D$ is  the pullback
by $q$ of the fibration $P _{
{f_2}}$ over $S ^{2k}$. By \eqref{degree}, a
section class $A
\in H ^{S}_2 (X)$ contributes to $ c ^{q} (P _{f})$ only if $c
_{ \text {vert}}
(A) \leq 0$; moreover, if $c _{\vert} (A)=0$, the class $A$ only contributes to
the degree zero class $c ^{q} _{0} (A)$ and so is not relevant to us. When $c
_{\vert} (A) <0$, the monotonicity
of $M$ implies that $ - \mathcal {C} (A)>0$ in this case,
because $X \simeq M \times S^2$ (since $f$ and $f _{h}$ map into components of
$Q$ corresponding to contractible loops in $ \Ham(M, \omega)$ by
construction). Put an admissible family $ \{\Omega_b\}$ on $P _{f_2}$ as in
\fullref{section.families}, so that the area of the fiber $X _{b}$ over $b \in S ^{2k}$ is
\begin{equation*} L ^{+} ( f (b)) + 2 \epsilon
\end{equation*}
with $\epsilon < - \mathcal {C}(A)$. Let $ \{J_b\}$ be a
compatible regular  family. The proof of 
\fullref{lower bound} implies that the area of each fiber of $p\co  P _{
f_2} \to S ^{2k}$ is at least $ - \mathcal {C} (A)$ whenever
there is a $ \{J_b\}$--holomorphic $A$--curve in that fiber. Thus, no element of
the moduli space $\mathcal {M} ^{*}_0 (P _{ f_2}, A; \{J_b\})$ lies in the fiber
over $s_0$, since the area of $\Omega _{s_0}$ is $2\epsilon$.

Pullback by $q$ the families $ \{\Omega_b\}, \{J_b\}$ to $P_{f_h}$ over $
\bar{D}$. The restriction of $ \{ q ^{*} \Omega_b\}$ over $\partial \bar {D}$
is by construction the constant family restricting to a split symplectic form,
ie $\omega+ \pi ^{*} (\alpha)$, with area $2\epsilon$ on each
fiber, since $\Omega
_{s_0}$ has that property.

Since $ \hat{f}$ is the constant map to $\id$ on $D ^{c}$,
the family $ \{q ^{*}\Omega_b\}$ over $\bar{D}$ can be extended to a family $ \{
\wwtilde{\Omega}_b\}$ on $P_{f_h}$ such that the area of each fiber $X$ over
$D ^{c}$ is $2\epsilon$. To see this note that  the fibers of $P _{f_h}|
_{D_c}$ can be identified with the product $M \times S^2$, up to an action of $S
^{1}$ which rotates the base $S^2$ and fixes $M$. Since the constant
family $ \{ q ^{*} \Omega_b\}$ over $\partial \bar {D}$ restricts to a split
form on the fibers $X \simeq M \times S^2$, which is invariant under this $S^1$
action, there is an extension $\{
\wwtilde{\Omega}_b\}$ of $ \{q ^{*} \Omega_b\}$ to $D ^{c}$. Pick any
extension $ \{ \tilde{J}_b\}$ of $ \{ q ^{*} (J_b)\}$ which is compatible with $\{
\wwtilde{\Omega}_b\}$.
By the above discussion, there are no $ \{
\tilde{J}_b\}$--holomorphic $A$--curves over $D ^{c}$. Thus,
$\{\tilde{J}_b\}$ is regular, since it is regular for curves over $D$ as it
is a pullback of a regular family $ \{J_b\}$ there. Moreover, $q$
pushes forward the moduli space
$ \mathcal {M} ^{*}_{0,1} (P_{f_h}, A;  \{\tilde{J}_b\})$ to the moduli
space $\mathcal {M} ^{*}_{0,1} (P _{f_2}, A, \{J_b\})$ ie  the diagram
\begin{equation*} \xymatrix{ \mathcal {M} ^{*}_{0,1} (P _{f_h}, A; \{
\widetilde{J_b}\}   \ar [r]^-{u \mapsto \tilde{q} \circ u} \ar [d]^ {}
& \mathcal {M} ^{*}_{0,1} (P _{f_2}, A; \{J_b\}) \ar [d]^- {}\\ \mathbb{CP} ^{k} 
\ar [r]^- {} & S ^{2k}}
\end{equation*}
commutes, where $ \tilde{q}$ is a lift of $q$ which is defined on $P
_{f_h}| _{D}$.

By definition, 
\begin{align*} c ^{q} (P _{f_2})=\sum_ {{A}} b _{{A}} \otimes
e ^{A} \in QH_* ^{S ^{2k}} ( M),\\
 c ^{q} (P _{f_h})=\sum_ {{A}} b' _{{A}} \otimes e ^{A} \in
QH ^{ \mathbb{CP} ^{k}}_* (M),
\end{align*}
where $b_A$ is the transverse intersection of $$\ev\co  \mathcal M ^{*}_
{0,1}(P _{f_h}, A, \{ \tilde{J}_b\}) \to P _{f_h}$$ with $I _{0} (
\mathbb{CP} ^{k}\times M),$
and $b'_A$  is the transverse intersection of $$\ev\co \mathcal M ^{*}_ {0,1}(P
_{f_2}, {A}, \{
J_b\}) \to P _{f_2}$$ with $I _{0} (S ^{2k} \times M)$. Since the above moduli
spaces lie over contractible subsets of $ \mathbb{CP} ^{k}$ and $ S ^{2k}$
\begin{align*}  
b _{A}= [\pt] \otimes b^{M}_A \in H_* (S ^{2k} \times M) \\ 
{b'}_A ^{M}= [\pt] \otimes {b'} _{A} ^{M} \in H_* ( \mathbb{CP} ^{k} \times M),
\end{align*}
for some $b_A, {b'}_A \in H_* (M)$. 
The above discussion implies that $b _{A} ^{M}={b'} ^{M}_{A}$.
Thus the two total classes are the same. To conclude that the only
nonvanishing classes of the two fibrations are in the top dimension note that
$ [\pt] \otimes b_A ^{M}$ has a nontrivial intersection pairing with $c \otimes
b \in H_* (S ^{2k} \times M)$ only if $c = [S ^{2k}] \in H _{*} (S ^{2k})$ and
use definition of the classes.
\qed
\begin{remark}  This proof makes  extensive use of monotonicity. It is
not
obvious to me if this theorem is true in a situation where one must use methods
of the virtual moduli cycle.  
\end{remark} 
\begin{lemma} \label {lemma3} The maps $f$ and $f_2$ above are freely homotopy
equivalent.
\end{lemma}
\begin{proof} 
 The map $f$ is induced from a composition of maps of 
 pairs
\begin{equation*}  (D ^{2k} \times S ^{1}, \partial D ^{2k} \times S ^{1}) \xrightarrow{t}
(S ^{2k+1}, \pt) \xrightarrow{ \hat{f}} (\Ham(M, \omega), \id).
\end{equation*}
On the other hand, ${f_2}$ is the induced map from the composition of
maps of pairs 
\begin{equation*}   (D ^{2k} \times S ^{1}, \partial D ^{2k} \times S ^{1})
\xrightarrow{i} (S ^{2k+1}, h ^{-1} ({D ^{c}}))
\xrightarrow{ \hat{f}} ( \Ham(M, \omega), \id).
\end{equation*}
Clearly, we can homotop $ \hat{f}$ through
maps of pairs to  a map $ \hat{f}'\co  (S ^{2k+1}, O ^{c}) \to (\Ham(M, \omega), \id)$, where $O \subset  h ^{-1} ({D})$ is an open ball
which does not contain $[\pt]$. Then $f$ is homotopic to a map induced from the
composition \begin{align*}  (D ^{2k} \times S ^{1}, \partial D ^{2k} \times S ^{1})
\xrightarrow{i} (S ^{2k+1}, O ^{c}) \xrightarrow{ \hat{f}'} ( \Ham(M, \omega), \id),
\end{align*}
and ${f_2}$ is induced from
\begin{align*}  (D ^{2k} \times S ^{1}, \partial D ^{2k} \times S ^{1})
\xrightarrow{t} (S ^{2k+1}, O ^{c}) \xrightarrow{ \hat{f}'} ( \Ham(M, \omega), \id).
\end{align*}
Thus, we just need to show that $i$ is homotopic via maps of pairs to $t$. To see
this one can use degree. 
\end{proof}
\section{The Hopf algebra structure of \texorpdfstring{$H_* ( \freels , \mathbb{Q})$}{H\137 *(LHam,Q)}} \label
{section.Hopf} This section is mostly an excursion, which studies the
relationship of the homomorphism $\Psi$ with the Pontryagin ring structure of $H_* ( \freels , \mathbb{Q})$. It
may be interesting to the reader in order to get an idea of how the use of
$S^1$--symmetry in the previous section relates to the bigger picture of QC
classes.

 The Milnor--Moore theorem states that a connected co-commutative
Hopf algebra $A$ over a  field of characteristic zero is generated by its
primitive elements. A  \emph{primitive} element is an element $a \in A$ such
that its coproduct is  $1 \otimes a + a\otimes 1$. More precisely it says that
$A$ is isomorphic as a  Hopf algebra to the universal enveloping algebra $
\mathcal {U} (P (A))$,  where $P (A)$ denotes the associated Lie algebra of its
primitive elements. In  other words the only relations in  $ \mathcal {U} (P
(A))$ are the ones of  the form $$a \otimes b - (-1) ^{pq} b \otimes a = ab -
(-1) ^{pq} ba,$$ where  the product on the right is the product in the Hopf
algebra. When $A$ is the  rational Hopf algebra of an $H$--space, Cartan--Serre
theorem states that the  Lie algebra of primitive elements consists of
spherical classes. In fact,  we have the following. \begin{theorem}[Milnor--Moore \cite{MM}, Cartan--Serre \cite{CS}] \label
{thm.milnor.moore} \label{Milnor-Moore}  Let $X$ be a connected 
 $H$--space. 
Denote by $ \pi_* (X, \mathbb{Q}) \subset H_* (X, \mathbb{Q})$ the  Lie
subalgebra  of the associated algebra of the ring, generated by the image of
the Hurewitz map $h\co  \pi_* (X) \to H_* (X, \mathbb{Q})$ and denote by
 $ \mathcal {U} (\pi_* (X, \mathbb{Q}))$ the universal enveloping algebra of  
 $\pi_* (X, \mathbb{Q})$. Then
\begin{equation*} H_* (X, \mathbb{Q}) \simeq \mathcal {U} (\pi_* (X, \mathbb{Q})),
\end{equation*}
as rings (in fact as Hopf algebras).
\end{theorem}
 For $ [\gamma] \in \pi_1 ( \Ham(M, \omega), \id) $, 
let $L^{[\gamma]} \subset \freels$ denote the component containing the loop 
$\gamma$. As a space $$L^{[\gamma]}=\Omega ^{[\gamma]} \Ham(M, \omega) \times \Ham(M, \omega),$$ where $\Omega ^{[\gamma]}
\Ham(M, \omega)$ denotes the $\gamma$--component of the based loop space
at $\id$. Hence, $$ \pi_* ( X ^{ [\gamma]}) \simeq \pi_* ( \Ham(M, \omega))
\oplus \pi_* (\Omega ^{ [\gamma]}  \Ham(M, \omega)).$$  Combining this
with \fullref{Milnor-Moore} ($L^{[\gamma]}$ is not a  connected $H$--space
naturally but is homeomorphic to one), we get  $$ H_* ( L^{[\gamma]},
\mathbb{Q}) \simeq \mathcal {U}  (\pi_* ( \Ham(M, \omega), \mathbb{Q}))
\otimes \mathcal {U}( \pi_* (\Omega ^{ [\gamma]}  \Ham(M, \omega)), \mathbb{Q}))
$$
as rings. By \fullref{vanishing} below, $ \Psi$ 
vanishes on $$H_* ( \Ham(M, \omega),
\mathbb{Q}) \simeq \mathcal {U} \left(\pi_*
( \Ham(M, \omega), \mathbb{Q}) \right)$$ for $*>0$. On the other hand
$\Psi (H_0 ( \Ham(M, \omega), \mathbb{Q})$ is generated over $ \mathbb{Q}$ by $[M]$, the multiplicative 
identity element; see \fullref{remark.seidelrep}.

 If one is to extend
$\Psi$ to a map $$\Psi\co  H_* (\freels, \mathbb{Q}) \to QH _{*+2n} (M),$$ which can
likely be done using pseudocycles, the above discussion together with \fullref{main
theorem} shows that
$\Psi$ would only be interesting on $$H_* (\ls, \mathbb{Q}) = \bigoplus
_{\gamma}\mathcal {U}( \pi_* (\Omega ^{[\gamma]} \Ham(M, \omega)) \subset
H_* (\freels, \mathbb{Q}),$$ a direct sum over $[\gamma]$ of free
graded commutative algebras on $ \pi_* (\ls)$. 
At the same time,
working on the free loop space allows us to pass to the $S^1$ equivariant
setting, using which we were able to do computations in \fullref{QC.and.Hofer}.

Define 
$i ^{ [\gamma]}\co  \Ham(M, \omega) \to L^{\gamma}$ to be 
the inclusion which
takes an element $\phi \in \Ham(M, \omega)$ to the loop $\phi \circ
\gamma$.
 \begin{lemma} \label{vanishing} If $k>0$, $\Psi (f)=0$ for $f\co 
B^{k} \to i ^{[\gamma]} ( \Ham(M,
\omega))$, where .
\end{lemma}
\begin{proof} This follows from the fact that for a map 
\begin{align*}  f\co  B ^{k} &\to
i_\gamma (\Ham(M, \omega))\subset L^{\gamma} \\
f (b)&= \phi_b \circ \gamma, \quad \text {where } \phi_b \in \Ham(M,
\omega),
\end{align*}
the fibration $P _{f}$ is
isomorphic to a  trivial $ \mathcal {F}$--fibration by an isomorphism which is a Hamiltonian
bundle map on each fiber,
and so the relevant invariants vanish by \mbox{\fullref{Axiom 1}}. Let 
$c _{ [\gamma]}$ be the constant map $f\co  B \to L ^{
[\gamma]}$ to the loop $ \gamma$. 
We trivialize $P _{f}$ as
follows: \[\displaylines{P _{f} = (B \times M \times D^2)_0 \cup _{f} (B
\times M \times D^2)_\infty\hfill
\cr
\hfill \xrightarrow{\tr} P _{c _{ [\gamma]}} = (B \times M \times D^2)'_0 \cup _{c
_{ [\gamma]}} (B \times M \times D^2)'_\infty,
\cr
\hfill\tr (b,x,z)_0 := (b, x,
z)'_0 \quad \text {and} \quad \tr (b, x, z)_\infty := (b, \phi_b ^{-1} (x),
z)'_\infty,\hfill
}\]
where $\phi_b \in \Ham(M, \omega)$ is as above. This map is
easily seen to be well defined. 
\end{proof}    
\section{Structure group of \texorpdfstring{$ \mathcal {F}$--}{F-}fibrations}  \label{structure group} This section is concerned with the
structure groups of the fibrations $ \tilde{p}\co  U\to \freels$ and $p\co U^{\smash{S^1}}  \to
Q$, which is indirectly used for the proof of \fullref{structures}. Another,
perhaps more pertinent goal here is to prove that $ \tilde{p}\co  U \to
\freels$ is universal for its structure group. 

 \begin{proposition} The
 structure group of ${p}\co  U\to \freels$ over the component of the loop
 $\gamma$ can be reduced to the group $ \mathcal {F} ^{\gamma}$ of Hamiltonian
 bundle maps of $X _{\gamma}$ which are identity over $D ^{2}_0$ and over a
 small neighborhood of $0 \in D ^{2}_\infty$ in coordinates of 
 \eqref{eq.X}.
\end{proposition}
This proposition follows immediately from \fullref{lemma.str.group} proved
later in this section. 
For $\gamma: S^1
\to \Ham(M, \omega)$ let $ [\gamma]$ denote its equivalence class in $
\pi_1 ( \Ham(M, \omega), \id)$. 
\begin{proposition}  \label{lemma.structure.group} Let $Q ^{ [\gamma]}$ denote
a connected component of $Q$.  The structure group of $p\co  U^{ [\gamma]} \to Q
^{ [\gamma]}$ may be reduced to the group
 of Hamiltonian bundle
maps of $\pi _\gamma\co  X _{\gamma} \to S^2$,  which sit over rotations
the base $S^2$, with the axis of rotation corresponding to $0 \in D ^{2}_0,
0 \in D ^{2}_\infty$. Moreover, elements of this group act as $\id \times \mathrm{rot}$
on $M \times D^2_0 \subset X _{\gamma}$ and by identity on the fiber over 
$0 \in D^2 _{\infty}$.
\end{proposition} 
The proof will be given after some preliminaries. To make the discussion more
transparent we work
with connections, which to us will be just smooth or continuous functors. In
fact, there is a natural such connection on $\tilde{p}\co  U \to
\freels$. 
\subsection {The path groupoid} A
\emph{topological category}
is a small category in which the set of all objects and the set of all morphisms
are topologized, so that the source and target maps
and all structure maps are continuous. Let $p\co  P \to B$ be a
bundle with fiber $X$, where $B$ is a topological group. Let $ \mathcal {C} (B)$
be a topological groupoid whose
objects are the points of $B$. The morphisms from $a$ to $b$ are defined to be
$$ \mathcal {C} (a, b)= P (a, b),$$ the space of
continuous paths from $a$ to $b$, ie maps $m\co  [0,1] \to B$ s.t.\ $ m (0)=a$ and $m
(1)=b$. 
The composition
law $$ \mathcal {C} (a,b) \times \mathcal {C} (b,c) \to \mathcal {C} (a,c)$$ is
defined as follows.
 Let $m _{a,b}\co  [0,1] \to B$ be a path with endpoints $a,b$ and
$m _{b,c}\co  [0,1] \to B$ be a path with endpoints $b,c$. Then $m _{b, c}
\circ m _{a,b} \co  [0,1] \to B$ is defined by
$$ m _{b,c} \circ m _{a, b} (t)= m _{b,c} (t) \cdot (m _{b,c} (0)) ^{-1} \cdot m
_{a,b} (t).$$ This path clearly has endpoints $a,c$ and is continuous. This is essentially the
only
natural way to define a composition law for paths in a topological group.
 The topology on the set of morphisms ie the free path space
of $B$ is taken to be the compact open topology.
\subsection {The category  \texorpdfstring{$\mathcal {D} (P,B,p)$}{D(P,B,p)}}
We also define a topological category $ \mathcal {D} (P,B,p)$, whose space of
objects is homeomorphic to $B$ with elements: manifolds
$X_b=p ^{-1} (b)$ for $b \in B$. The space of morphisms from $X_a$ to $X_b$ is
defined to be 
\begin{equation*} \mathcal {D} (X_a, X_b)= \text {Homeo} (X_a, X_b),
\end{equation*} 
the space of homeomorphisms from $X_a$ to $X_b$. The composition law is just the
composition of homeomorphisms. 
\subsubsection {Topology on the space of morphisms of  $\mathcal {D} (P,B,p)$}
For each $b \in B$, let $U_b \subset B$ be an open set with
a trivialization $\phi _{b}\co  U _{b} \times X \to p ^{-1} (U_b)$. Let now
$a, b \in B$. Any morphism 
whose source is the fiber $X _{u_1}$ with $u_1 \in U_a$ and target $X _{u_2}$
with $u _{2} \in U _{b}$ can be identified via the trivializations $\phi_a,
\phi_b$ with
a homeomorphism from $X$ to $X$. Thus, the set of such morphisms is
identified with $U_a \times U _b \times \text {Homeo} (X, X)$, which we will
denote by $ \mathcal {D} (U_a, U_b)$. It has a natural topology, where
the topology on $\text {Homeo} (X,X)$ is the compact-open topology. The basis
for a topology on the set of all morphism then consists of open sets 
 in $ \mathcal {D} (U_a, U_b)$ for all $a,b \in B$. Clearly, a different choice
of trivializations gives rise to equivalent topologies. 
\subsection {Connections}
\begin{definition} Let $p\co  P \to B$ be as above. An
\emph{abstract connection} is defined to be
a continuous functor $F$ from the
category $ \mathcal {C} (B)$ to $ \mathcal {D} (P,B,p)$.
\end{definition}

The map $F
(m)\co  p ^{-1} (m_0) \to p ^{-1} (m_1)$ will be called
the \emph{parallel transport map}. The name of the connection is the name of the
corresponding functor (eg $F$). The word abstract in abstract connection will
often be dropped. We may define the holonomy group of an abstract connection exactly the same way
as for usual smooth connections on $G$--bundles, using the parallel transport
maps.
\begin{lemma}
 \label{lemma.str.group} The
 structure group of $ {p}\co  P \to B$ over a connected component can be reduced to
 the holonomy group $\Hol(F)$ of the connection ${F}$ on this component. 
\end{lemma}
\begin{proof} 
Let $
\{U_i\}$ be a cover of $B$ by contractible open sets and  $H_i \co  U_i
\times I \to B$ be a free homotopy, which at time 0 is the constant map to
$b_0$ and at time 1 is the inclusion map of $U_i$. Then parallel translating by $F$, along the paths of the homotopy $h _{i,x} (t)=H_i (x, t)$, gives a trivialization
$\tr_i\co
U_i \times X \to p ^{-1} (U_i)$. The transition map $\tr _{ij}\co  U_i
\bigcap U_j \times X \to U_i
\bigcap U_j \times X$ is by construction and functoriality of $F$ 
given by
parallel translation by $F$ along the loops $ h ^{-1} _{j,x} \circ h _{i, x}$.
Here $\circ$ is the multiplication in the groupoid $ \mathcal {C} (B)$. 
\end{proof}
\subsection {A connection \texorpdfstring{${F _{U}}$}{F\137 U} on \texorpdfstring{${p}\co  U \to \freels$}{p: U -> LHam}} The space
$\freels$ is a topological group and we may take the topological groupoid
 $ \mathcal {C} (\freels)$ defined as above, except  that we take the
 morphisms in the groupoid to be smooth in the sense below.
 \begin{definition} \label{def.smooth} We define a map $m\co  [0,1] \to
 \freels $ to be \emph{smooth} if it is locally constant at the
 endpoints and the associated map $ \tilde{m}\co  [0,1] \times S^1 \to \text
 {Ham}(M, \omega)$ is smooth.
\end{definition}
 The groupoid  $ \mathcal {C} (\freels)$ is topologized as a subspace of
 continuous maps with its compact open topology.
 \subsubsection*
 {The parallel
 transport map} Let $m\co  I \to \freels$ be a path. We define the map ${F _{U}}
 (m)= t_m$ from the fiber $X _{m_0}$
over $m (0)=m_0$, to the fiber $X _{m_1}$ over $m (1)=m_1$ as follows.  We have 
\begin{align*} X _{m_0}&= M \times D^2_0 \cup _{m_0}M  \times D^2_\infty, \\
X_{m_1}&= M \times D^2_0 \cup_{m_1} M \times D^2_\infty
\end{align*} If $r, \theta$ are polar coordinates on $D
^{2}$, then \begin{equation*} 
t_m (x, r, \theta)_0= (x, r,
\theta)_0 \text \quad { \text {and}} \quad
t_m (x,r, \theta)_\infty= (m _{r, \theta} \circ m ^{-1} _{0, \theta}(x), r, \theta),
\end{equation*}
where $m _{r, \theta}$ denotes the element of the loop $m _{r}=m (r)$ at time
$\theta$.

This is well defined under the gluing since the
diagram $$ \xymatrix{(x, 1, \theta)_0   \ar [r]^-{\hbox{\footnotesize$\sim$}} \ar [d]^
-{t_m} & (m _{0, \theta} (x), 1, \theta) _\infty \ar  [d]^-
{t_m} \\ (x, 1, \theta)_0  \ar [r]^-{\hbox{\footnotesize$\sim$}} & \bigl (m _{1, \theta} \circ m
^{-1} _{0, \theta
} \circ m _{0, \theta} (x)=m _{1, \theta} (x), 1, \theta \bigr)_\infty} $$
commutes. 
We leave it to the reader to verify that this gives 
a continuous functor ${F _{U}}\co  \mathcal {C} (\freels) \to
\mathcal {D} ({U}, \freels,{p})$, which assigns to $\gamma \in \freels$ the fiber
$X_\gamma$ and
to a morphism $m\co  I \to \freels$ from $\gamma_0$ to $\gamma_1$ the map $t_m\co X
_{\gamma_0} \to X _{\gamma_1}$. 
We denote by
 $\freels ^{\gamma}$ the component of the loop $\gamma$ in $\freels$.

\begin{lemma}  The group $\Hol (F _{U})$ is
isomorphic to the group $ \mathcal {C} (\gamma, \gamma)$ of automorphisms of the
object $\gamma$ in $ \mathcal {C} (\freels)$.
\end{lemma}
\begin{proof} By
construction of the connection
 $F _{U}$,
the natural surjective holonomy map $\hol\co   \text {Aut} _{
\mathcal {C}} (\gamma) \to \Hol (F)$ has no kernel.
\end{proof}
Let $E$ denote the space of all smooth paths in $\freels$ based at $\gamma$
(see \fullref{def.smooth}). This is a contractible space with a free 
continuous action of the group $\text {Aut} _{ \mathcal {C}} (\gamma)$ acting by left multiplication using the topological 
groupoid structure of $ \mathcal {C} (\freels)$. Moreover, this action
fixes the fibers of the
projection $k\co  E \to \freels$ given by evaluating at the endpoint and is
transitive on the fibers. It follows that $k\co  E \to \freels$ is the universal $\text
{Aut} _{ \mathcal {C}} (\gamma)$--bundle. In other words
\begin {equation*} B\text {Aut} _{
\mathcal {C}} (\gamma) = B \Hol (F _{U})= \freels ^{\gamma}.
\end {equation*}
In fact, we have the following: 
\begin{proposition} The bundle $ \tilde{p}\co  U \to \freels ^{\gamma}$ is the
associated bundle to the universal principal $\Aut _{ \mathcal {C}} (\gamma)$--bundle $k\co  E \to
\freels ^{\gamma}$.
\end{proposition} 
\begin{proof} This follows from the proof of \fullref{lemma.str.group}. The
details are left to the reader.
\end{proof}

\subsection[Proof of \ref{lemma.structure.group}]{Proof of \fullref{lemma.structure.group}} Recall that the
bundle $p\co  U^{\smash{S^1}} \to Q$ is the quotient by the $S^1$ action $ \tilde{\rho}$ of
the bundle $p \times \id\co  U \times S ^{\infty} \to \freels \times S ^{\infty}$
(cf \fullref{setup}).
Let $V_i$ be a contractible open set in $Q ^{\gamma}$ and $$g_i: V_i \times S^1 
\to \freels \times S^\infty$$ a local trivialization of the principal $S^1$ bundle
$h\co  \freels \times S^\infty \to Q ^{\gamma}$. Let $H_i$ be a free homotopy
of the map $g_i\co 
{V_i \times 0} \to \freels \times S^\infty$ to the constant map to
$(\gamma_0,s_0 )$. As before, the connection $F$ then induces a map $$t_i\co  V_i
\times X _{\gamma_0} \to {p} ^{-1} ( g_i (U_i \times 0)),$$ by parallel
translating along the paths of the homotopy $H_i$. 
The transition maps $t _{ij}$ have the form:
\begin{equation*} t _{ij} (u, x) =  (u, t ^{-1}_j \circ \tilde{\rho}(\theta
_{ij}) ^{-1} \circ t_i (x))
\end{equation*}
where $\theta _{ij}$ comes from the transition maps $g _{ij}\co V _i
\cap V_j \times S^1 \to V _i \cap V_j
\times S^1$, $g _{ij} (u, \theta) = (u, \theta + \theta _{ij})$.  By
construction, this is a Hamiltonian bundle map of
$\pi _{\gamma}\co  X_\gamma \to S^2$ to itself which sits over the rotation by
$\theta _{ij}$
of the base and fixes the fibers over $0$ and infinity.
\qed   
\bibliographystyle{gtart}
\bibliography {link}  
\end{document}